\documentclass[12pt,a4paper,english, reqno]{amsart}
\usepackage[english]{babel}
\usepackage{amsthm}
\usepackage{amsmath,color}
\usepackage{amssymb}
\usepackage{graphicx}
\usepackage{dsfont}
\usepackage[bf]{caption}
\usepackage{hyperref}
\usepackage{parskip}

\newtheorem{thm}{Theorem}[section]

\newtheorem{lem}[thm]{Lemma}
\newtheorem{prop}[thm]{Proposition}

\newtheorem{remark}[thm]{Remark}

\newtheorem{defi}[thm]{Definition}
\theoremstyle{rmk}

\newcommand{\e} {\varepsilon}
\newcommand{\p} {\textnormal{\textsf{P}}}
\newcommand{\E} { \textnormal{\textsf{E}}}
\newcommand{\PP} { \mathbb{P} }

\newcommand{\N} { \mathbb{N} }

\newcommand{\Z} { \mathbb{Z} }
\newcommand{\R} { \mathbb{R} }

\newcommand{\var}{\mathop{\mathsf{Var}}}

\allowdisplaybreaks

\begin{document}

\title[Collisions of several walkers in recurrent random environments]
{Collisions of several walkers in recurrent random environments}
\date{\today}

\author{Alexis Devulder}
\address{Laboratoire de Math\'ematiques de Versailles, UVSQ, CNRS, Universit\'e Paris-Saclay,
78035 Versailles, France.}
\email{devulder@math.uvsq.fr }

\author{Nina Gantert}
\address{Technische Universit\"at M\"unchen,
Fakult\"at f\"ur Mathematik,
85748 Garching,
Germany }
\email{gantert@ma.tum.de}

\author{Fran\c{c}oise P\`ene}
\address{Universit\'e de Brest and Institut Universitaire de France,
LMBA, UMR CNRS 6205, 29238 Brest cedex, France}
\email{francoise.pene@univ-brest.fr}

\maketitle

\begin{abstract}

We consider $d$ independent walkers on $\Z$, $m$ of them performing
simple symmetric
random walk and $r= d-m$ of them performing recurrent RWRE (Sinai walk), in $I$ independent random environments. We show that the product is recurrent, almost surely, if and only if $m\leq 1$ or $m=d=2$.
In the transient case with $r\geq 1$, we  prove that the walkers meet infinitely often, almost surely, if and only if $m=2$ and $r \geq I= 1$. In particular, while $I$ does not have an influence for the recurrence or transience, it does play a role for the probability to have infinitely many meetings. To obtain these statements, we prove two subtle localization results for a single walker in a recurrent random environment, which are of independent interest.

\end{abstract}

\section{Introduction and statement of the main results}
Recurrence and transience of products of simple symmetric random walks on $\Z^d$ is well-known since the works of P\'olya \cite{Polya}.
If the product of several walks is transient, one may ask if they meet infinitely often. It is also well-known and goes back to Dvoretzky and Erd\"os, see  (\cite{Dvoretzky_Erdos}, p. 367)
that $3$ independent simple symmetric random walks (SRW) in dimension $1$ meet infinitely often almost surely
%(see also \cite{BarlowPeresSousi}),
while $4$ walks meet only finitely often, almost surely. In fact,  P\'olya's original interest in recurrence/transience of simple random walk came from a question about collisions of two independent walkers on the same grid, see \cite {PolyaInc}, ``Two incidents''.

The classical topic of meetings/collisions of two or more walkers walking on the same graph has found recent interest, see \cite{KrishnapurPeres}, \cite{BarlowPeresSousi}, where the grid is replaced by more general graphs.
It is well-known that if a graph is recurrent for simple random walk, two independent walkers do not necessarily meet infinitely often, see \cite{KrishnapurPeres}. Since on a {\sl transitive} recurrent graph,
two independent walkers do meet infinitely often, almost surely, see \cite{KrishnapurPeres},  the ``infinite collision property'' describes how far the recurrent graph is from being transitive.
For motivation from physics, see \cite{CampariCassi}.

We investigate this question for products of recurrent random walks in random environment (RWRE)
and of simple symmetric random walks on $\Z$. It is known already that, for any $n$, a product of $n$ independent RWRE
in $n$ i.i.d. recurrent random environments is recurrent, see \cite{Z01}, and that $n$ independent walkers in the same recurrent random environment meet infinitely often in the origin,
see \cite{NMFa}.
Here, we consider several walkers each one performing either a Sinai walk or a simple symmetric random walk, with the additional twist that not all Sinai walkers are necessarily using the same environment.
%We emphasize that simple symmetric random walks can be seen as degenerated Sinai walks.
%in one-dimensional random/non-random environment, with the additional twist that not all walkers are necessarily using the same environment.
%\af{More precisely, the RWRE considered are Sinai walks, whereas the random walks in non random environments are simple symmetric random walks, which are random walks in a degenerated Sinai environment}.

Let $d,m,r$ be nonnegative integers such that $m+r=d\ge 1$.
We consider $d$ walkers, $m$ of them performing SRW $S^{(1)},...,S^{(m)}$ and the $r$ others
performing random walks $Z^{(1)},...,Z^{(r)}$ in $I$ independent random environments, with $I\le r$.
More precisely, we consider $r$ collections
of i.i.d. random variables $\omega^{(1)}:=\big(\omega^{(1)}_x\big)_{x\in\mathbb Z},\dots,
\omega^{(r)}:=\big(\omega^{(r)}_x\big)_{x\in\mathbb Z}$, taking values in $(0,1)$ and
 defined on the same probability space $(\Omega,\mathcal F, \p)$,  such that
$\omega^{(1)},...,\omega^{(I)}$
are independent and such that the others are exact copies of some of these $I$ collections,
i.e., for every $j\in\{I+1,...,r\}$, there exists an index $J_j\in\{1,...,I\}$ such that $\omega^{(j)}\equiv\omega^{(J_j)}$.
A realization of $\omega:=\left(\omega^{(1)},...,\omega^{(r)}\right)$
will be called an {\it environment}. Recall that we denote by
$$
\p \text{ the law of the environment } \omega.
$$
We set
$$
    Y_n
:=
    \big(S_n^{(1)},...,S_n^{(m)},Z_n^{(1)},...,Z_n^{(r)}\big),
\qquad
    n\in\N,
$$
and make the following assumptions.
Given $\omega=\left(\omega^{(1)},...,\omega^{(r)}\right)$ and
$x\in \mathbb Z^d$, under $P_\omega^x$, $S^{(1)},...,S^{(m)},Z^{(1)},...,Z^{(r)}$
are independent Markov chains such that $P_\omega^{x}(Y_0=x)=1$ and
for all $y\in\Z$ and $n\in\N$,
\begin{equation}\label{eqDefQuenchedLaw}
    P_{\omega}^x\big[S^{(i)}_{n+1}=y+1\big|S_n^{(i)}=y\big]
=
    \frac 12
=
    P_{\omega}^x\big[S^{(i)}_{n+1}=y-1\big|S_n^{(i)}=y\big],
\quad
   i\in\{1,...,m\},
   %y\in\Z,\ n\in\N,
\end{equation}
\begin{equation}\label{eqDefQuenchedLaw_2}
    P_{\omega}^x\big[Z^{(j)}_{n+1}=y+1\big|Z_n^{(j)}=y\big]
=
    \omega^{(j)}_{y}
=
    1-P_{\omega}^x\big[Z^{(j)}_{n+1}=y-1\big|Z_n^{(j)}=y\big],
\quad
    j\in\{1,...,r\}.
    %,y\in\Z,\ n\in\N.
\end{equation}
We set $S^{(i)}:=\big(S_n^{(i)}\big)_n$ and $Z^{(j)}:=\big(Z_n^{(j)}\big)_n$
for every $i\in\{1,...,m\}$ and every $j\in\{1,...,r\}$.
Note that, for every $j$, $Z^{(j)}=\big(Z_n^{(j)}\big)_n$ is a {\it random walk on $\mathbb Z$ in the environment} $\omega^{(j)}$, and that the $S^{(i)}$'s are independent
SRW, independent of the $Z^{(j)}$'s and of their environments.
We call $P_\omega:=P_\omega^0$ the {\it quenched law}.
Here and in the sequel we write $0$ for the origin in  $\mathbb Z^d$.
We also define the {\it annealed law} as follows:
$$
    \mathbb P[\cdot]:=\int P_\omega[\cdot]\p(\text{d}\omega).
$$
Setting $\rho_k^{(j)}:=\frac{1-\omega_k^{(j)}}{\omega_k^{(j)}}$ for
$j\in\{1,...,r\}$ and $k\in\mathbb Z$, we assume moreover that there exists
$\varepsilon_0\in(0,1/2)$ such that for every $j\in\{1,...,r\}$,
\begin{equation}\label{eqHypothesesSinai1}
    \p\big[\omega_0^{(j)}\in[\varepsilon_0,1-\varepsilon_0]\big]=1,
\qquad
    \E\big[\log \rho_0^{(j)}\big]=0,
\qquad
    \sigma_j^2:= \E\big[(\log \rho_0^{(j)})^2\big]>0,
\end{equation}
where $\E$ is the expectation with respect to $\p$.
Under these assumptions, the $Z^{(j)}$ are RWRE, often called {\it Sinai's walks} due to the famous
result of \cite{S82}. Solomon \cite{S75} proved the recurrence of
$Z^{(j)}$ for $\p$-almost every environment.
We stress in particular that the assumption $ \sigma_j^2 > 0$ excludes the case of deterministic environments, hence when we say ``Sinai's walk'', we always refer to a random walk in a ``truly'' random environment.

Our first result concerns the recurrence/transience of $Y:=(Y_n)_n$. Recurrence of $Y$
means that $S^{(1)},....,S^{(m)},Z^{(1)},...,Z^{(r)}$ meet simultaneously
at 0 infinitely often. As explained previously, this result is known for SRW (i.e. if $m=d$) since \cite{Polya}
and more recently for RWRE (i.e. if $r=d$, that is, if $m=0$)
in the case where the environments $\omega^{(j)}$ are independent (i.e. $I=r=d$, see \cite{Z01,NMFa}) and in the case where the environment $\omega^{(j)}$ is the same for all the RWRE (i.e. $r=d,I=1$, see \cite{NMFa}). See also \cite{Gallesco} for related results.

\medskip

\begin{thm}\label{TheoremProductRW}
If $m\le 1$, or if $m=d=2$, then, for $\p$-almost every $\omega$,
the random walk $Y$ is recurrent with respect to $P^0_\omega$.
Otherwise, for $\p$-almost every $\omega$, the random walk $Y$ is transient with respect to $P^0_\omega$.
\end{thm}

\medskip

In particular, a product of two recurrent RWRE and one SRW is recurrent, while a product of
two SRW and one recurrent RWRE is transient.

When $Y$ is transient, a natural question is the study of the
simultaneous meetings (i.e., collisions)   of
$S^{(1)},....,S^{(m)},Z^{(1)},...,Z^{(r)}$. That is, we would like to extend the results of
\cite{Polya,Dvoretzky_Erdos} to the case in which some of the random walks are in random environments (when $r\geq 1$).
We recall that when $r=0$, the number of collisions is, by \cite{Polya,Dvoretzky_Erdos}, almost surely infinite
if $m\leq 3$ and almost surely finite when $m\geq 4$.
Interestingly, compared to Theorem \ref{TheoremProductRW}, the behaviour depends on
whether $I=1$  (when the RWRE are all in the same environment)
or $I\ge 2$ (at least two RWRE are in independent environments).

\medskip

\begin{thm}\label{simultaneousmeeting}
We distinguish the $3$ following different cases.
\begin{itemize}
\item[(i)]
If $m\ge 3$ and $r\ge 1$, then, for $\p$-almost every environment $\omega$,
$$
     P^0_\omega\big[S_n^{(1)}=S_n^{(2)}=S_n^{(3)}=Z_n^{(1)}\mbox{ infinitely often}\big]
=
    0,
$$
i.e. almost surely, the walks
$S^{(1)},S^{(2)},S^{(3)},Z^{(1)}$
meet simultaneously only a finite number of times. A fortiori,
$S^{(1)},\dots ,S^{(m)},Z^{(1)}, \dots Z^{(r)}$
also meet simultaneously only a finite number of times.
\item[(ii)] If $m=2$ and $r\geq I=1$, then for $\p$-almost every environment $\omega$,
$$
     P^0_\omega\big[S_n^{(1)}=S_n^{(2)}=Z_n^{(1)}=...=Z_n^{(r)}\mbox{ infinitely often}\big]
=
    1,
$$
i.e. almost surely, the walks
$S^{(1)},S^{(2)},Z^{(1)},...,Z^{(r)}$
meet simultaneously infinitely often.
\item[(iii)] If $m=2$ and $r\geq I\ge 2$, then for $\p$-almost every environment $\omega$,
$$
    P^0_\omega\big[S_n^{(1)}=S_n^{(2)}=Z_n^{(1)}=Z_n^{(2)}\mbox{ infinitely often}\big]
=
    0,
$$
i.e. almost surely, the walks
$S^{(1)},S^{(2)},Z^{(1)},Z^{(2)}$,
and a fortiori the walks
$S^{(1)},S^{(2)}$, $Z^{(1)},...,Z^{(r)}$,
meet simultaneously only a finite number of times.
\end{itemize}
\end{thm}
This last result can be summarized in the following manner. Assume that $r\geq 1$ and that $Y$ is transient (i.e. $m\ge 2$ and $r\ge 1$), then
$S^{(1)},...,S^{(m)},Z^{(1)},...,Z^{(r)}$ meet simultaneously infinitely often  if and only if $m=2$
and $I=1$.
Hence our results cover collisions of an arbitrary number of random walks in equal or independent random
(or deterministic) recurrent environments.

\begin{remark} The results of Theorem \ref{simultaneousmeeting} remain true if
the simple random walks are replaced by random walks on $\mathbb Z$ with i.i.d. centered increments with finite and strictly positive variance. However, we write the proof of this theorem only in  the case of SRW to keep the proof more readable and less technical.
\end{remark}

The case of transient RWRE in the same
subballistic random environment is investigated in \cite{DGP18} (in preparation).

In order to demonstrate Theorem \ref{simultaneousmeeting}, we prove
the two following propositions.
The first one deals with two independent recurrent RWRE in two independent environments.
%The following proposition plays a crucial role in the proof of the case (ii) of Theorem \ref{simultaneousmeeting} when $r>1$ and $I=1$.
\begin{prop}\label{Lemmeannexe2}
Assume $r\geq I\geq 2$.
For every $\varepsilon>0$, $\mathbb P\big[Z_n^{(1)}=Z_n^{(2)}\big]
=O\left((\log n)^{-2+\varepsilon}\right)$.
\end{prop}
The second proposition deals with $r$ independent recurrent RWRE in the same environment.
\begin{prop}\label{Lemmeannexe1}
Assume $r> I=1$.
%and let $Z^{(1)},...,Z^{(r)}$ be $r$ independent recurrent RWRE in the same environment $\omega$ satisfying \eqref{eqHypothesesSinai}.
For $\p$-almost every $\omega$, there exists $c(\omega)>0$ such that, for every
$(y_1,...,y_r)\in[(2\Z)^r\cup (2\Z+1)^r]$, we have
$$
    \limsup_{N\to+\infty}\frac{1}{\log N}
    \sum_{n=1}^N\frac 1n\sum_{k\in\mathbb Z}\prod_{j=1}^rP_\omega^{y_j}[Z_n^{(j)}=k]
\ge
    c(\omega).
    %\ \mbox{infinitely often}.
$$
\end{prop}
These two propositions are based on two new localization results for recurrent RWRE,
which are of independent interest.
These two localization results use the {\sl potential} of the environment (see \eqref{eqDefPotentialV}) and its {\sl valleys}, these quantities were introduced by Sinai in \cite{S82} and are crucial for the investigation of the RWRE.

In the first one, stated in Proposition \ref{LemmaLocalisationVerticale} and used to prove Proposition \ref{Lemmeannexe2},
we localize a recurrent RWRE at time $n$ with (annealed) probability $1-(\log n)^{-2+\varepsilon}$ for $\varepsilon>0$, whereas
previous localization results for such RWRE were with probability $1-o(1)$
(see \cite{S82}, \cite{Golosov84}, \cite{KTT}, \cite{Bovier_Faggionato} and \cite{Freire}),
or with probability $1-C\big(\frac{\log\log\log n}{\log \log n}\big)^{1/2}$ for some $C>0$
(see \cite{Andreoletti_Alternative}, eq. (2.23)), and they localize the RWRE inside one valley.
In order to get our more precise localization probability, which is necessary to apply the Borel-Cantelli
lemma in the proof of Item (iii) of Theorem \ref{simultaneousmeeting},
we localize the RWRE
in an area of low potential defined with several valleys instead of just one. To this aim, we
study and describe typical trajectories of the recurrent RWRE into these different valleys.
% with the required probability.

In our second localization result, stated in Proposition \ref{Lemmeannexe1Bis} and used to prove Proposition \ref{Lemmeannexe1},
we prove that for large $N\in\N$, with high probability on $\omega$ (for $\p$),
the quenched probability $P_\omega[Z_n=b(N)]$ is larger than a positive constant,
%for any $n\in[\frac N2,N]$, where $b(N)$
uniformly for any even $n\in[N^{1-\e},N]$ for some $\e>0$, where $b(N)$
is the (even) bottom of some valley of the potential $V$ of a recurrent RWRE $Z$ (defined in \eqref{eqDefWidehatbN}).
In order to get this uniform probability estimate, we use a method different from that of previous localization results,
based on a coupling between recurrent RWRE.

The article is organized as follows. In Section \ref{SinaiEstimate},
we give an estimate on the return probability of recurrent RWRE, see Proposition \ref{MAJO}, which is of independent interest.
Our main results for direct products of walks are proved
in Section \ref{sec:product}. The proofs concerning the simultaneous meetings
of random walks are based on the above-mentioned two key localization results for recurrent RWRE, proved in Sections \ref{independentenv}
and \ref{sameenv}.
%They use the {\sl potential} of the environment, see \eqref{eqDefPotentialV} and its {\sl valleys}, these quantities were introduced by Sinai in %\cite{S82} and are crucial for the investigation of the RWRE.
%The recurrence/transience of the multidimensional random walk in random
%environment $M$ is established in Section \ref{multidim}.
%The main result of Section \ref{independentenv} is a localization result:
%with probability $1-(\log n)^{-2+\varepsilon}$ a recurrent RWRE is in a
%area of low potential defined with several valleys (see Proposition \ref{LemmaLocalisationVerticale}).
%In Section \ref{sameenv}, we prove that, with high probability on $\omega$ (for $\p$),
%the quenched probability $P_\omega[|Z_n-b(N)|\le 1]$ is larger than a positive constant,
%%for any $n\in[\frac N2,N]$, where $b(N)$
%for any $n\in[N^{1-\e},N]$ for some $\e>0$, where $b(N)$
%is the bottom of some valley of the potential $V$ of the recurrent RWRE $Z$ (see Proposition \ref{Lemmeannexe1Bis}).

%
%
%
%
%
%
\section{A return probability estimate for the rwre}\label{SinaiEstimate}
We consider a recurrent one dimensional RWRE $Z=(Z_n)_n$ in the random environment $\omega=(\omega_x)_{x\in\mathbb Z}$,
where the $\omega_x\in(0,1)$, $x\in\Z$,  are i.i.d.
(that is, $Z_0=0$ and \eqref{eqDefQuenchedLaw_2} is satisfied with $Z$ and $\omega$ instead of $Z^{(j)}$ and $\omega^{(j)}$).
 We assume the existence
of $\varepsilon_0\in(0,1/2)$ such that
\begin{equation}\label{eqHypothesesSinai}
    \p[\omega_0\in[\varepsilon_0,1-\varepsilon_0]]=1,
\qquad
    \E[\log \rho_0]=0,
\qquad
    \E[(\log \rho_0)^2]>0,
\end{equation}
where $\rho_k:=\frac{1-\omega_k}{\omega_k}$, $k\in\Z$.
The following result completes \cite[Theorem 1.1]{NMFa} which says that,
for every $0\leq\vartheta<1$, we have for $\p$-almost every environment $\omega$,
$$
    \sum_{n\ge 1} \frac{P_\omega^0[Z_n=0]}{n^\vartheta}
=
    \infty.
$$
\begin{prop}\label{MAJO}
For $\p$-almost every environment $\omega$,
$$
    \sum_{n\ge 1}
    \frac{P_\omega^0[Z_n=0]}n<\infty.
$$
\end{prop}
Before proving this result, we introduce some more notations. First, let
$$
    \tau(x)
:=
    \inf\{n\ge 1\ :\ Z_n=x\},
\qquad
    x\in\Z.
$$
In words, $\tau(x)$ is the hitting time of the site $x$ by the RWRE $Z$.
As usual, we consider the {\it potential} $V$, which is a function of the environment $\omega$ and is defined on $\Z$ as follows:
\begin{equation}\label{eqDefPotentialV}
    V(x)
:=
    \left\{
    \begin{array}{lr}
        \sum_{i=1}^x \log\frac{1-\omega_i}{\omega_i}
    &
        \textnormal{if } x>0,
    \\
        0
    &
        \textnormal{if }x=0,
    \\
        -\sum_{i=x+1}^0 \log\frac{1-\omega_i}{\omega_i}
    &
        \textnormal{if } x<0.
    \end{array}
    \right.
\end{equation}
The potential is useful since it relates to the description of the RWRE as an electric network. It can be used to estimate ruin probabilities for the RWRE. In particular,
we have  (see e.g. \cite[(2.1.4)]{Z01} and \cite[Lemma 2.2]{Devulder_Persistence} coming from \cite[p. 250]{Z01}),
\begin{eqnarray}
\label{probaatteinte}
    P_\omega^b[\tau(c)< \tau(a)]
& = &
    \bigg(\sum_{j=a}^{b-1} e^{V(j)}\bigg)\bigg(\sum_{j=a}^{c-1} e^{V(j)}\bigg)^{-1},
\qquad
    a<b<c
\end{eqnarray}
and, recalling  $\varepsilon_0$ from \eqref{eqHypothesesSinai1} and \eqref{eqHypothesesSinai},
\begin{eqnarray}
\label{InegEsperance1}
    E_\omega^b[\tau(a)\wedge \tau(c)]
& \leq &
    \varepsilon_0^{-1}(c-a)^2
    \exp\Big[\max_{a\leq \ell \leq k \leq c-1; k\ge b}\big(V(k)-V(\ell)\big)\Big],
\qquad
    a<b<c\, ,
\end{eqnarray}
where $E_\omega^b$ denotes the expectation with respect to $P_\omega^b$
and $u\wedge v:=\min(u,v)$, $(u,v)\in\R^2$.
For symmetry reasons, we also have
\begin{eqnarray}
\label{InegEsperance2}
    E_\omega^b[\tau(a)\wedge \tau(c)]
& \leq &
    \varepsilon_0^{-1}(c-a)^2
    \exp\Big[\max_{a\leq \ell \leq k \leq c-1,\ \ell\le b-1}\big(V(\ell)-V(k)\big)\Big],
\quad
    a<b<c\, .
\end{eqnarray}
Moreover, we have, for $k\ge 1$ (see Golosov  \cite{Golosov84}, Lemma 7)
\begin{eqnarray}
\label{InegProba1}
    P_\omega^b[\tau(c)<k]
& \leq & k \exp\left(\min_{\ell \in [b,c-1]}V( \ell)-V(c-1)\right),\qquad     b<c\, ,
\end{eqnarray}
and by symmetry,
%replacing each $\omega_x$ by $1-\omega_{-x}$,
we get  (similarly as in Shi and Zindy \cite{ShiZindy}, eq. (2.5) but with some slight differences for the values of $\ell$)
%and (see Shi and Zindy \cite{ShiZindy}, eq. (2.5))
%\begin{eqnarray}
%\label{InegProba2}
%    P_\omega^b[\tau(a)<k]
%& \leq & k \exp\left(\min_{\ell \in [a+1,b]}V(\ell) -V(a+1)\right),\qquad     a<b\, .
%\end{eqnarray}
%{\bf In fact I obtain instead the following formula, replacing $\omega_i$ by $1-\omega_{-i}$ and $Z$ by $-Z$,
%similarly as in Shi and Zindy \cite{ShiZindy}, eq. (2.5) but with slight index differences:}
\begin{eqnarray}
\label{InegProba2}
    P_\omega^b[\tau(a)<k]
& \leq & k \exp\left(\min_{\ell \in [a,b-1]}V(\ell) -V(a)\right),\qquad     a<b\, .
\end{eqnarray}

\smallskip

\begin{lem}
Let $\gamma>0$. For $\p$-almost every $\omega$, there exists $N(\omega)$ such that for every $n\ge N(\omega)$,
$$
    n^{\frac 12-\gamma}
\le
    \max_{k\in\{0,\dots, n\}} V(k)
\le
    n^{\frac 12+\gamma},
\qquad
    -n^{\frac 12+\gamma}
\le
    \min_{k\in\{0,\dots, n\}} V(k)
\le
    -n^{\frac 12-\gamma},
$$
and such that the same inequalities hold with
$\{-n,\dots, 0\}$
instead of
$\{0,\dots, n\}$.
\end{lem}
\begin{proof}
Observe that it is enough to prove that $\p$-almost surely,
\begin{equation}\label{Ineg}
    n^{\frac 12-\gamma}
\leq
    \max_{1 \leq k \leq n} V(k)\le n^{\frac 12+\gamma}
\end{equation}
if $n$ is large enough
(up to a change of $\log \rho_i$ in $-\log \rho_i$, in $\log \rho_{1-i}$ or in
$-\log \rho_{1-i}$).
%\eqref{Ineg} follows from the law of the iterated logarithm for $\max_{1 \leq k \leq n} V(k)$, see e.g. \cite{Re1990}, p. 31.
The first inequality of \eqref{Ineg} is given by \cite[Theorem 2]{H65}.
The second inequality of \eqref{Ineg} is  a consequence of the law of iterated logarithm for $V$,
as explained in (\cite{Chung01}, end of p. 248).
\end{proof}
\begin{proof}[Proof of Proposition \ref{MAJO}]
Let $\eta\in(0,1)$ and $n\geq 2$.
 We define
$$
    z_+
:=
%    \inf\{y\ge 1\ :\ V(y)\le -\eta\log\log n\},
    \inf\{y\ge 1\ :\ V(y)\le -(\log n)^{1-\eta}\},
\qquad
    z_-
:=
%    \sup\{y\le -1\ :\ V(y)\le -\eta\log\log n\}.
    \sup\{y\le -1\ :\ V(y)\le -(\log n)^{1-\eta}
\}.
$$
Due to the previous lemma, choosing $\gamma$ small enough, we have that $\p$-almost surely, if $n$ is large enough, the following inequalities hold:
\begin{equation}\label{EEE0}
    |z_\pm|
\le
    %\frac{(\log\log n)^{2+\varepsilon}}2
    \frac{(\log n)^{2-\eta}}{2}
\mbox{ and }
    \max_{z_-\le i,j\le z_+} (V(i)-V(j))
\le
    %(\log\log n)^{1+\varepsilon}.
    (\log n)^{1-\eta/10}.
\end{equation}
We have by the strong Markov property,
\begin{multline}
\label{IneqProbaReturnZero}
    P_\omega^0[Z_n=0]
\le
    P_\omega^0[\tau(z_+)>n,\ \tau(z_-)>n]
    +
    \sum_{k=0}^{n}P_\omega^0[\tau(z_+)=k]P_\omega^{z_+}[Z_{n-k}=0]
\\
    +
    \sum_{k=0}^{n}P_\omega^0[\tau(z_-)=k]P_\omega^{z_-}[Z_{n-k}=0].
\end{multline}
Recall that, given $\omega$, the Markov chain $Z$ is an electrical network
where, for every $x\in\Z$,  the conductance of the bond $(x,x+1)$ is $C_{(x,x+1)}=e^{-V(x)}$
(in the sense of Doyle and Snell \cite{Doyle_Snell}).
In particular, the reversible measure $\mu_\omega$
(unique up to a multiplication by a constant) is given by
\begin{equation}\label{reversiblemeas}
    \mu_\omega(x)
:=
    e^{-V(x)}+e^{-V(x-1)},
\qquad
    z\in\Z.
\end{equation}
 So we have
\begin{eqnarray*}
P_\omega^{z_\pm}[Z_{n-k}=0]&=&P_\omega^0[Z_{n-k}=z_\pm]\frac{\mu_\omega(0)}
     {\mu_\omega(z_\pm)}\le \frac{\mu_\omega(0)}
     {\mu_\omega(z_\pm)}=\frac{e^{-V(0)}+e^{-V(-1)}}{e^{-V(z_\pm)}+e^{-V(z_\pm-1)}}
\\
&\le&
    \frac{e^{-V(0)}+e^{-V(-1)}}{e^{-V(z_\pm)}}\le \big(e^{-V(0)}+e^{-V(-1)}\big)
%    (\log n)^{-\eta}.
    \exp\big[-(\log n)^{1-\eta}\big].
\end{eqnarray*}
Hence,
\begin{equation}\label{EEE1}
    \sum_{k=0}^{n}P_\omega^0[\tau(z_\pm)=k]P_\omega^{z_\pm}[Z_{n-k}=0]
\le
    \big(e^{-V(0)}+e^{-V(-1)}\big)
    %\frac 1{(\log n)^{\eta}}
    \exp\big[-(\log n)^{1-\eta}\big].
\end{equation}
%since $\eta>1$.
Moreover we have due to \eqref{InegEsperance1} and to Markov's inequality,
\begin{eqnarray*}
    P_\omega^0[\tau(z_+)>n,\ \tau(z_-)>n]
& \leq &
    n^{-1}
    E_\omega^0[\tau(z_+)\wedge \tau(z_-)]
\\
&\le&
    n^{-1}
    \varepsilon_0^{-1}
    (z_+-z_-)^2\exp\Big[\max_{z_-\le \ell\le k\le z_+-1}[V(k)-V(\ell)]\Big].
\end{eqnarray*}
%due to Lemma 2.1 in \cite{ShiZindy}.
Now using \eqref{EEE0}, $\p$-almost surely,
we have
$$
    P_\omega^0[\tau(z_+)>n,\ \tau(z_-)>n]
\le
    \varepsilon_0^{-1} n^{-1} (\log n)^{4-2\eta}\exp\big[(\log n)^{1-\eta/10}\big]
$$
for every $n$ large enough. This combined with \eqref{IneqProbaReturnZero}, \eqref{EEE1} and $e^{-V(-1)}\leq \e_0^{-1}$ gives
$\p$-almost surely for large $n$
\begin{equation*}
%\sum_{n\ge 1} \frac{P_\omega^0(\tau(z_+)>n,\ \tau(z_-)>n)}n<\infty
    P_\omega^0[Z_n=0]
\le
    5 \varepsilon_0^{-1}\exp\big[-(\log n)^{1-\eta}\big].
%    (\log n)^{-\eta}.
\end{equation*}
Consequently,
$
\sum_{n\geq 1}\frac{P_\omega^0[Z_n=0]}{n}<\infty
$
$\p$-almost surely,
which ends the proof of Proposition \ref{MAJO}.
\end{proof}
\section{Direct product of Walks}\label{sec:product}
We start with a proof of Theorem \ref{TheoremProductRW}.
With a slight abuse of notation, we will write $0$ for the origin in $\Z^k$, whatever $k$ is.
\begin{proof}
%\begin{enumerate}
1. If $m\geq 1$ and $r=0$, then $(Y_n)_n$ is a product of $m$ independent simple random walks on $\Z$.
It is well-known that it is recurrent if $m\in\{1,2\}$, and transient if $m\geq 3$. This follows from elementary calculations and
the crucial fact that for any irreducible Markov chain $(G_n)_n$,
\begin{equation}\label{MCrecurr}
(G_n)_n \text{ is recurrent if and only if }\sum_{n\ge 0} P^x [G_n=0] = \infty,
\end{equation}
where $x$ is one of the states of the Markov chain.\\
2. If $m\geq 3$ and $r\geq 1$,
then the 3-tuple of the three first coordinates  of $(Y_n)_n$ is
$\big(S_n^{(1)},S_n^{(2)},S_n^{(3)}\big)_n$
which is a product of $3$ independent simple random walks on $\mathbb Z$, hence is transient. So  $(Y_n)_n$ is transient for
$\p$-almost every $\omega$.\\
3. If $m= 2$ and $r\geq 1$, then applying the local limit theorem (see e.g. Lawler and Limic \cite{Lawler_Limic} Prop. 2.5.3)
for $S^{(1)}$ and $S^{(2)}$ for $n\in\N^*$,
$$
    P_{\omega}^0 [Y_n=0]
=
    \prod_{i=1}^2 P\big[S_{n}^{(i)}=0\big]
    \prod_{j=1}^r P_{\omega^{(j)}}^0\big[Z_{n}^{(j)}=0\big]
\leq
    \frac{c}{n}P_{\omega^{(1)}}^0\big[Z_{n}^{(1)}=0\big],
$$
where $c>0$ is a constant.
This and Proposition \ref{MAJO} yield
$
   \sum_{n=0}^\infty P_{\omega}^0 [Y_n=0]
<
\infty
$
for $\p$-almost $\omega$. Hence, (using the Borel-Cantelli Lemma or \eqref{MCrecurr}), $(Y_n)_n$ is $\p$-almost surely transient.
\\
4. We now assume $m \in \{0,1\}$.
We choose some $\delta\in(0,1/5)$ such that
$3\delta r<\frac{1-2\delta}2$.
We denote by $\lfloor x\rfloor$ the integer part of $x$ for $x\in\R$.
%Let $L$ be large enough so that $e^{(1-2\delta)L}\ge 8 e^{3\delta L}$.
For $L\in\N$, we have
\begin{eqnarray*}
    \sum_{n\ge 0} P_{\omega}^0 [Y_n=0]
& \geq &
    \sum_{n=\big\lfloor \frac{e^{(1-2\delta)L}}2\big\rfloor+1}^{\lfloor e^{(1-2\delta)L}\rfloor}
    P_{\omega}^0 [Y_{2n}=0]
\\
& = &
    \sum_{n=\big\lfloor \frac{e^{(1-2\delta)L}}2\big\rfloor+1}^{\lfloor e^{(1-2\delta)L}\rfloor}
    P\big[S_{2n}=0\big]^m
    \prod_{j=1}^r P_{\omega^{(j)}}^0\big[Z_{2n}^{(j)}=0\big].
\end{eqnarray*}
Due to \cite{NMFa} (Propositions 3.2, 3.4 and
 (3.22)), since $\delta\in(0,1/5)$, there exist
$C(\delta)>0$ and a sequence $(\Gamma(L,\delta))_{L\in\N}$ of elements of $\mathcal F$
(that is, depending only on $\omega$)
such that
\begin{equation}
\label{eqProbaLimsupGamman1}
    \p\left[\bigcap_{N\ge 0}\bigcup_{L\ge N}\Gamma(L,\delta)\right]
=
    1
\end{equation}
and such that, for every $L\in\N$, on  $\Gamma(L,\delta)$, we have
\begin{equation}\label{InegProbaZeroSinaionGamma1}
    \forall i\in\{1,\dots,r\},
    \, \forall k_i\in\big\{\lfloor e^{3\delta L}\rfloor+1,\cdots,\lfloor e^{(1-2\delta)L}\rfloor\big\},
\ \
    P_{\omega^{(i)}}^0\big[Z_{2k_i}^{(i)}=0\big]
\ge
    C(\delta)\, e^{-3\delta L}.
\end{equation}
Due to the local limit theorem, this gives on $\Gamma(L,\delta)$,
for large $L$ so that $\frac{e^{(1-2\delta)L}}2\geq e^{3\delta L}$,
\begin{eqnarray*}
&\ &
    \sum_{n\ge 0} P_{\omega}^0 [Y_n=0]
\geq
    \frac{e^{(1-2\delta)L}}3
    \left(\frac{c}{e^{(1-2\delta)L/2}}\right)^m
    \bigg(\frac{C(\delta)}{e^{3\delta L}}\bigg)^r
\geq
    c_1(\delta) e^{[(1-2\delta)/2-3\delta r]L},
\end{eqnarray*}
which goes to infinity as $L$ goes to infinity due to our choice of $\delta$, $c_1(\delta)$ being a positive constant.
Thanks to \eqref{eqProbaLimsupGamman1}, this gives
$\sum_{n\ge 0} P_{\omega}^0 [Y_n=0]=+\infty$ for $\p$-almost all $\omega$.
Consequently, due to \eqref{MCrecurr}, $(Y_n)_n$ is recurrent for $\p$-almost every environment $\omega$.
%\end{enumerate}
\end{proof}

\begin{remark}
Recall that Sinai \cite{S82} (see also Golosov \cite{Golosov84}) proved the convergence in distribution
of $\big(Z_n^{(i)}/(\log n)^2\big)_n$. Recall also that, due to de Moivre's theorem, $\big(S_n^{(i)}/\sqrt{n}\big)_n$
converges in distribution. Due to Theorem \ref{TheoremProductRW},
$Y$ is recurrent iff $\sum_n 1/(n^{\frac m2}((\log n)^2)^{r})=\infty$,
where $n^{\frac m2}((\log n)^2)^{r}$ is the product of the normalizations
of the coordinates of $Y$ under the (non Markovian) annealed law $\PP$.
\end{remark}

%The question whether two independent simple random walks in $\Z^d$ meet infinitely often was
%studied in 1921
%by P{\'o}lya in his celebrated paper about recurrence and transience of simple random walks \cite{Polya},
%where meeting means being at the same place at the same time.
%He proved that two independent simple random walks almost surely meet infinitely often in $\Z^2$,
%but almost surely meet only a finite number of times in $\Z^3$.
%In dimension one, it is known (see Dvoretzky and Erd{\"o}s  \cite{Dvoretzky_Erdos}, Section 7)  that
%\begin{eqnarray*}
%    \mathbb P^0[S_n^{(1)}=S_n^{(2)}=S_n^{(3)}\ \mbox{infinitely often}]
%& = &
%    1,
%\\
%    \mathbb P^0[S_n^{(1)}=S_n^{(2)}=S_n^{(3)}=S_n^{(4)}\mbox{ infinitely often}]
%& =
%&
%    0.
%\end{eqnarray*}

Note also that Theorem \ref{simultaneousmeeting} and the previous paragraph lead to the following statement
(only for $r\geq 1$):
if $\sum_{n\geq 1} \frac{1}{n^{m/2}(\log n)^{2I-2}}<\infty$,
then almost surely, $S_n^{(1)},\dots ,S_n^{(m)},Z_n^{(1)}, \dots Z_n^{(r)}$
meet simultaneously only a finite number of times;
otherwise, they almost surely meet simultaneously infinitely often.

Now we will start to prove Theorem \ref{simultaneousmeeting}.
Note that the case $m\leq 1$ is already treated in Theorem \ref{TheoremProductRW}
which says that in this case the random walks meet infinitely often at $0$.
%Moreover, the case $r=0$ is already known, as said before.

\begin{proof}[Proof of Theorem \ref{simultaneousmeeting}]
Let $A_n:=\big\{S_n^{(1)}=...=S_n^{(m)}=Z_n^{(1)}=...=Z_n^{(r)}\big\}$ for $n\geq 0$.\smallskip\\
{\it Proof of (i).\/} Assume $m= 3$ and $r= 1$. Observe that for large $n$,
\begin{eqnarray*}
    \mathbb P^0_\omega[A_n]
& = &
    \sum_{k\in\mathbb Z} \mathbb P^0_\omega\big[Z_n^{(1)}=k\big]\big(\mathbb P\big[S_n^{(1)}=k\big]\big)^3
\\
& \leq &
     C\sum_{k\in\mathbb Z} \frac{\mathbb P^0_\omega\big[Z_n^{(1)}=k\big]}{n\sqrt{n}}
=
    \frac{C}{n\sqrt{n}}
\end{eqnarray*}
for some $C>0$ since
for every $k\in\Z$ and $n\in\N$,
$
    \mathbb P\big[S_{2n}^{(1)}=k\big]
\leq
    \mathbb P\big[S_{2n}^{(1)}=0\big]
\sim_{n\to+\infty}
    (\pi n)^{-1/2}
$
due to the local limit theorem.
Hence
$
    \sum_n  \mathbb P^0_\omega\big[S_n^{(1)}=S_n^{(2)}=S_n^{(3)}=Z_n^{(1)}\big]
<
    \infty
$ and (i) follows by the Borel-Cantelli lemma in this case and a fortiori when $m\geq 3$ and $r\geq 1$.
\smallskip\\
{\it Proof of (ii).\/}
Assume $m=2$ and $r\geq I=1$.
Since $I=1$, all the RWRE are in the same environment, which is necessary to apply
Proposition \ref{Lemmeannexe1}, which is essential to prove (ii).
We use the generalization of the second Borel Cantelli lemma due to Kochen and Stone \cite{KochenStone}
combined with a result by Doob. To simplify notations, we also write $\omega$ for $\omega^{(1)}$,
so $\omega^{(i)}=\omega$ for every $1\leq i \leq r$.

We first prove that $\sum_n\mathbb P_\omega[A_n]=\infty$ a.s. More precisely,
we fix an initial condition $x=(x_1,x_2,y_1,...,y_r)\in(2\mathbb Z)^{2+r}\cup(2\mathbb Z+1)^{2+r}$.
We have for all $n$ and $\omega$,
\begin{eqnarray*}
    P^{x}_\omega[A_{n}]
& = &
    \sum_{k\in\mathbb Z}
        \mathbb P\big[x_1+S_n^{(1)}=k\big]\mathbb P\big[x_2+S_n^{(2)}=k\big]
        \prod_{j=1}^r P_{\omega}^{y_j}\big[Z_n^{(j)}=k\big].
%\\
%& \ge &
%    \sum_{|k|\le (\log n) ^3, k-(x_1+n)\in(2\Z)}
%        \mathbb P\big[x_1+S_n^{(1)}=k\big]\mathbb P\big[x_2+S_n^{(2)}=k\big]
%        \prod_{j=1}^r P_{\omega}^{y_j}\big[Z_n^{(j)}=k\big].
\end{eqnarray*}
Notice that, for every $i\in\{1,2\}$, due to the de Moivre-Laplace theorem (see e.g.
\cite[Prop. 2.5.3 and Corollary 2.5.4]{Lawler_Limic},
%  or   \cite[VII-3]{Feller1}),
$$
   \sup_{k \in (x_i+n+2\mathbb Z),\ |k|\le(\log n )^3} \left|\mathbb P\big[x_i+S_n^{(i)}=k\big]
    -\frac{\sqrt{2}}{\sqrt{\pi n}} e^{-(k-x_i)^2/(2n)}
   \right|=o(n^{-1/2}).
$$
Consequently for large even $n$, for every $\omega$,
$$
    P^{x}_\omega[A_{n}]
\geq
    \sum_{|k|\le (\log n) ^3, k-(x_1+n)\in(2\Z)}
    \frac{1}{\pi n}
    \prod_{j=1}^r P_{\omega}^{y_j}\big[Z_n^{(j)}=k\big]
=
    \frac{1}{\pi n}
    \sum_{|k|\le (\log n) ^3}
    \prod_{j=1}^r P_{\omega}^{y_j}\big[Z_n^{(j)}=k\big].
$$
This remains true for large odd $n$.
%e.g. by (Lawler et al. \cite{Lawler_Limic} Corollary 2.5.4).
Hence for large $n$,
\begin{equation}\label{InegPAn}
    P^{x}_\omega[A_{n}]
\geq
    \frac{1}{\pi n} P_\omega^{(y_1,...,y_r)}\big[Z_n^{(1)}=...=Z_n^{(r)}\big]
    -
    \frac{1}{\pi n} P_\omega^{y_1}\big[\big|Z_n^{(1)}\big|>(\log n)^3\big].
\end{equation}
%\begin{eqnarray*}
%    P^{x}_\omega[A_{n}]
%& = &
%    \sum_{k\in\mathbb Z}
%        \mathbb P\big[x_1+S_n^{(1)}=k\big]\mathbb P\big[x_2+S_n^{(2)}=k\big]
%        \prod_{j=1}^r P_{\omega}^{y_j}\big[Z_n^{(j)}=k\big]
%\\
%& \ge &
%    \sum_{|k|\le (\log n) ^3, k-(x_1+n)\in(2\Z)}
%        \left( \prod_{j=1}^r P_{\omega}^{y_j}\big[Z_n^{(j)}=k\big]\right)
%    \prod_{i=1}^2\left(\frac{\sqrt{2}e^{-\frac{(k-x_i)^2}{2n}}}{\sqrt{\pi n}}-\frac{C}{n}\right)
%\\
%& \ge &
%    \sum_{|k|\le (\log n) ^3}  \left(\prod_{j=1}^r
%    P_{\omega}^{y_j}\big[Z_n^{(1)}=k\big]\right)
%    \left(\frac{2e^{-\frac{k^2}{n}}}{\pi n}-\frac{C_{x_1,x_2}} {n\sqrt{n}}\right)
%\\
%& \ge &
%    \left(\sum_{|k|\le (\log n) ^3}  \left(\prod_{j=1}^r
%    P_{\omega}^{y_j}\big[Z_n^{(1)}=k\big)]\right)
%    \frac{2e^{-\frac{k^2}{n}}}{\pi n}\right)-\frac{C_{x_1,x_2}} {n\sqrt{n}}
%\end{eqnarray*}
%for some $C>0$ and $C_{x_1,x_2}>0$ due to the local limit theorem (see e.g. Lawler and Limic \cite{Lawler_Limic} Prop. 2.5.3).
Recall that $\big(Z_n^{(1)}/(\log n)^3\big)_n$ converges almost surely to 0
with respect to the annealed law (see \cite{DeheuvelsRevesez} Theorem 4, or more recently \cite{HuShi} Theorem 3).
This holds also true for $P_\omega^{y_1}$ for $\p$-almost every $\omega$, so the last probability in \eqref{InegPAn}
goes to $0$ as $n\to+\infty$,
which yields
$
    \lim_{N\to+\infty}\frac{1}{\log N}\sum_{n=1}^N \frac{1}{n} P_\omega^{y_1}\big[\big|Z_n^{(1)}\big|>(\log n)^3\big]
=
    0
$.
%Hence, writing $e^{-\frac{k^2}{n}} = 1 - (1 - e^{-\frac{k^2}{n}})$,
%$$    P^{x}_\omega[A_{n}]
%\ge    \left(\sum_{|k|\le (\log n) ^3}  \left(\prod_{j=1}^r
%    P_{\omega}^{y_j}[Z_n^{(1)}=k]\right)
%    \frac{2}{\pi n}\right)-\frac{2(\log n)^6}{\pi n^2}-\frac{C_{x_1,x_2}} {n\sqrt{n}}.$$
Hence for $\p$-almost every $\omega$,
\begin{equation}\label{InegLimsupSumPAn}
    \limsup_{N\rightarrow +\infty}
    \frac 1{\log N}\sum_{n=1}^ N{P_\omega^x[A_n]}
\ge
    \frac {c(\omega)}{\pi},
\end{equation}
with
$
    c(\omega)
:=
    \inf_{(y_1,...,y_r)\in[(2\Z)^r\cup (2\Z+1)^r]}\limsup\limits_{N\rightarrow+\infty}\frac 1{\log N}
    \sum_{n=1}^N \frac{1}{n} P_\omega^{(y_1,...,y_r)}\big[Z_n^{(1)}=...=Z_n^{(r)}\big]
$.
If $r=1$, then $c(\omega)=1$. If $r>1$, due to Proposition \ref{Lemmeannexe1},
$c(\omega)>0$ for $\p$-almost every environment $\omega$. This implies that
\begin{equation}\label{sumPAn}
\sum_{n\ge 1}P_\omega^x[A_n]=+\infty.
\end{equation}
Moreover, let $C>0$ be such that for all $n\geq 1$ and $k\in\Z$,
$
    \mathbb{P}\big[S_n^{(1)}=k\big]
\leq
    C n^{-1/2}
$, which exists e.g. since
$
    \mathbb{P}\big[S_{2n}^{(1)}=k\big]
\leq
    \mathbb{P}\big[S_{2n}^{(1)}=0\big]
\sim_{n\to+\infty}
    (\pi n)^{-1/2}
$
by the local limit theorem.
So for $1\leq n<m$, we have by Markov property,
\begin{eqnarray*}
&&
    P^{x}_\omega[A_{n}\cap A_{m}]
\\
& = &
    \sum_{(k,\ell)\in\mathbb Z^2}
    P_\omega^{(y_1,...,y_r)}\big[Z_n^{(1)}=...=Z_n^{(r)}=k,\ Z_m^{(1)}=...=Z_m^{(r)}=\ell\big]
\\
& \  &
    \qquad \qquad \quad \ \times
    \mathbb P\big[x_1+S_n^{(1)}=k\big]
    \mathbb P\big[x_2+S_n^{(2)}=k\big]
    \big(\mathbb P\big[S_{m-n}^{(1)}=\ell-k\big]\big)^2
\\
& \le &
    \sum_{k\in\Z}
    P_\omega^{(y_1,...,y_r)}\big[Z_n^{(1)}=...=Z_n^{(r)}=k\big]
    P_\omega^{(k,...,k)} \big[Z_{m-n}^{(1)}=...=Z_{m-n}^{(r)}\big]
    \frac{C^4}{n(m-n)}
    %\frac{4(1+o(1))}{\pi^2 n(m-n)}
\\
& \le &
    \frac{C^4}{n(m-n)}.
    %\frac{4(1+o(1))}{\pi^2 n(m-n)}.
\end{eqnarray*}
Consequently, for large $N$,
$$
    \sum_{1\le n,m\le N, m\neq n}P_\omega^{x}[A_{n}\cap A_{m}]
\leq
    2\sum_{n=1}^N\frac{C^4}{n}\sum_{\ell=1}^{N-n}\frac{1}{\ell}
\leq
    3C^4(\log N)^2.
%\leq
%    8\pi^{-2}[1+o(1)](\log N)^{2}.
$$
Applying this and \eqref{InegLimsupSumPAn} we get for $\p$-almost every $\omega$, for every initial condition $x\in(2\Z)^{2+r}\cup (2\Z+1)^{2+r}$,
\begin{equation}\label{sumPAnAm}
\limsup_{N\rightarrow+\infty}
    \frac{\left(\sum_{n=1}^N P_\omega^{x}[A_n]\right)^2}{\sum_{1\le n,m\le N}P_\omega^{x}[A_{n}\cap A_{m}]}
\geq
    \frac{(c(\omega))^2}{3\pi^2C^4}.
    %\frac{(c(\omega))^2}{2}.
\end{equation}
Due to the Kochen and Stone extension of the second Borel-Cantelli lemma (see Item (iii) of the main theorem of \cite{KochenStone} applied with $X_n=\sum_{i=1}^n\mathbf 1_{A_i}$, or \cite[p. 317]{Spitzer}),
\eqref{sumPAnAm} and \eqref{sumPAn} imply that
%$P_\omega^x[A_n\ i.o.]>(c(\omega))^2/2$.
$
    P_\omega^x[A_n\ i.o.]
=
    P_\omega^x\big[\cap_{N\geq 0}\cup_{n\geq N}A_n\big]
\ge
    (c(\omega))^2/(3\pi^2C^4)>0
$,
where i.o. means infinitely often.
Now for $\p$-almost every $\omega$, due to a result by Doob (see for example Proposition V-2.4 in \cite{Neveu}),
since $E:=\{A_n\ i.o.\}=\cap_{N\geq 0}\cup_{n\geq N}A_n$ is invariant
(with respect to the shifts of the sequence $(Y_0, Y_1, Y_2, \ldots $)),
for every $x\in(2\mathbb Z)^{2+r}\cup(2\mathbb Z+1)^{2+r}$,
$\big(P_\omega^{(S^{(1)}_n,S^{(2)}_n,Z^{(1)}_n,...,Z^{(r)}_n)}[E]\big)_n$
converges $P_\omega^x$-almost surely to $\mathbf 1_E$.
But $\inf_{x\in(2\mathbb Z)^{2+r}\cup(2\mathbb Z+1)^{2+r}}  P_\omega^x[E]\geq (c(\omega))^2/(3\pi^2C^4)>0$,
so we conclude that $\mathbf 1_E=1$ $P_\omega^x$-almost surely,
thus $P_\omega^x(E)=1$, for $\p$-almost every environment $\omega$.
\medskip\\
{\it Proof of (iii).\/}  Assume $m=2$ and $r=I=2$.
We have
\begin{eqnarray*}
    \mathbb P^0[A_n]
&=&
    \sum_{k\in\mathbb Z} \mathbb P\big[Z_n^{(1)}=Z_n^{(2)}=k\big]\big(\mathbb P\big[S_n^{(1)}=k\big]\big)^2
\\
&\le &
    \frac{C^2}{n}\mathbb P\left[Z_n^{(1)}=Z_n^{(2)}\right]
=
    O\big(n^{-1}(\log n)^{-3/2}\big),
\end{eqnarray*}
due to Proposition \ref{Lemmeannexe2} and the local limit theorem.
Hence $\sum_n\mathbb P^0[A_n]<\infty$ and (iii) follows due to the Borel-Cantelli lemma.
\end{proof}

So there only remains to prove Propositions \ref{Lemmeannexe2} and \ref{Lemmeannexe1}.

\section{Probability of meeting for two independent recurrent rwre in independent environments}\label{independentenv}
The aim of this section is to prove Proposition \ref{Lemmeannexe2}, which is a key result in the proof of case (iii) of Theorem \ref{simultaneousmeeting}.

Let $Z^{(1)}$ and $Z^{(2)}$ be two independent recurrent RWRE in independent
environments $\omega^{(1)}$ and $\omega^{(2)}$ satisfying \eqref{eqHypothesesSinai}.

The main idea of the proof is that $Z_n^{(1)}$ and $Z_n^{(2)}$
are localized with high (annealed) probability in two
areas (depending on the environments, see Proposition \ref{LemmaLocalisationVerticale})
which have no common point with high probability (see Lemma \ref{LemmaIntersectionBottoms2IndependentPotentials}).
Due to \cite{S82}, we know that, with high probability, $Z_n^{(i)}$
is close to the bottom $B_n^{(i)}$ of some valley (containing 0 and of height larger than $\log n$) for the potential $V^{(i)}$.
Here and in the following, $V^{(i)}$ denotes the potential corresponding to $\omega^{(i)}$,
defined as in \eqref{eqDefPotentialV} with $\omega$ replaced by $\omega^{(i)}$.
An intuitive idea to prove Proposition
\ref{Lemmeannexe2} should then be that
$
    p_n
:=
    \mathbb P\left[\max_{i=1,2}\big|Z_n^{(i)}-B_n^{(i)}\big|\geq \big|B_n^{(1)}-B_n^{(2)}\big|/2\right]
$
is very small. More precisely we would like to prove that $p_n = O\big((\log n)^{-1-\varepsilon}\big)$.
(In view of the proof of (iii) above, it would suffice to show that
$\sum_n \frac{p_n}{n}<\infty$).
However, this seems difficult to prove and we are not even sure that it is true.
Indeed, in view of Lemma \ref{LemmaCardinalHTiPetits} below (proved for a continuous approximation $W^{(i)}\approx V^{(i)}$),
we think that with probability greater than $ 1/\log n$, $0$ belongs to a valley of height between $\log n - 2\log \log n$ and $\log n$ and that
%as a consequence,
the annealed probability that $Z_n^{(i)}$ is close to the bottom of this valley (which is not $B_n^{(i)}$)
should be greater that $1/\log n$. Hence, to prove Proposition \ref{Lemmeannexe2}, we will work with several valleys instead of a single one.
\subsection{Proof of Proposition \ref{Lemmeannexe2}}
In this subsection, we use a Brownian motion $W^{(i)}$, approximating the potential $V^{(i)}$, to  build a localization domain $\Xi_n\big(W^{(i)}\big)$
for $Z_n^{(i)}$, $i\in\{1,2\}$. This localization is stated in Proposition \ref{LemmaLocalisationVerticale} and is crucial to prove Proposition \ref{Lemmeannexe2}.

In order to construct our localization domain $\Xi_n\big(W^{(i)}\big)$, we use the notion of $h$-extrema, defined as follows.

\begin{defi}[\cite{NP}]
If $w:\R\rightarrow \R$ is a continuous function
 and $h>0$, we say that $y_0\in\mathbb R$ is an {\it $h$-minimum} for $w$
if there exist real numbers $a$ and $c$ such that
$a<y_0<c$,
$w(y_0)=\inf_{[a, c]} w$,
$w(a)\geq w(y_0)+h$
and
$w(c)\geq w(y_0)+h$.
We say that $y_0$ is an {\it $h$-maximum} for $w$ if $y_0$ is an $h$-minimum for $-w$.
In any of these two cases, we say that $y_0$ is an {\it $h$-extremum} for $w$.
\end{defi}

We also use the following notation.

\begin{defi}\label{calW}
As in \cite{Cheliotis}, we denote by $\mathcal{W}$ the set of functions $w$ : $\R\rightarrow \R$ such that
the three following conditions are satisfied:
{\bf (a)} $w$ is continuous on $\R$;
{\bf (b)} for every $h>0$, the set of $h$-extrema of $w$ can be written
$\{x_k(w,h),\ k\in\Z\}$, with $(x_k(w,h))_{k\in\Z}$ strictly increasing, unbounded from below and above,
and with $x_0(w,h)\leq 0 < x_1(w,h)$, notation that we use in the rest of the paper on $\mathcal{W}$;
{\bf (c)} for all $k\in\Z$ and $h>0$, $x_k(w,h)$ is an $h$-minimum for $w$ if and only if $x_{k+1}(w,h)$ is an $h$-maximum for $w$.
\end{defi}
We now introduce, for $w\in\mathcal W$, $i\in\Z$ and $h>0$,
$$
    b_{i}(w,h)
:=
    \left\{
    \begin{array}{ll}
      x_{2i}(w,h)  & \text{if } x_0(w,h) \text{ is an  $h$-minimum},\\
      x_{2i+1}(w,h) &  \text{otherwise}.
    \end{array}
    \right.
$$
As a consequence, the $b_{i}(w,h)$ are the $h$-minima of $w$.
We denote by $M_i(w,h)$ the unique $h$-maximum of $w$ between  $b_{i}(w,h)$ and $b_{i+1}(w,h)$.
That is, $M_i(w,h)=x_{j+1}(w,h)$ if $b_{i}(w,h)=x_j(w,h)$.
%Hence, $b_{w,i}(h)$ is  the minimum of $w$ in the interval $[M_{i-1}(w,h), M_i(w,h)]$, which we call the $i$-th valley of $w$ of height at least $h$.
%Similarly, $M_i(w,h)$ is the maximum of $w$ in the interval $[b_{i-1}(w,h), b_i(w,h)]$.

For $w\in\mathcal W$, $h>0$ and $i\in\Z$, the restriction of $w-w(x_i(w,h))$ to $[x_i(w,h),x_{i+1}(w,h)]$
is denoted by $T_i(w,h)$ and is called an {\it $h$-slope}, as in \cite{Cheliotis}.
If $x_i(w,h)$ is an $h$-minimum (resp. $h$-maximum),
then $T_i(w,h)$  is a nonnegative (resp. nonpositive) function, and its maximum (resp. minimum) is attained at $x_{i+1}(w,h)$.
We also introduce, for each slope $T_i(w,h)$, its {\it height}
$H(T_i(w,h)):=|w(x_{i+1}(w,h))-w(x_i(w,h))|\geq h$,
and its {\it excess height} $e(T_i(w,h)):=H(T_i(w,h))-h\geq 0$.

When $x_i(w,h)$ is an $h$-minimum, the restriction of
$w$ to $[x_{i-1}(w,h),x_{i+1}(w,h)]$ will sometimes be called
{\it valley} of height at least $h$ and of bottom $x_i(w,h)$. The height
of this valley is defined as $\min\{w(x_{i-1}(w,h)),w(x_{i+1}(w,h))\}-w(x_i(w,h))$, which can also be rewritten $\min\{H(T_{i-1}(w,h)),H(T_{i}(w,h))\}$.

These $h$-extrema are  useful to localize RWRE and diffusions in a random potential.
Indeed, a diffusion in a two-sided Brownian potential $W$ (resp. in a ($-\kappa/2$)-drifted Brownian potential $W_\kappa$ with $0<\kappa<1$)
is localized at large time $t$ with high probability in a small neighborhood of $b_{0}(W,\log t)$
(resp. some of the $b_{i}(W_\kappa,\log t-\sqrt{\log t})$, $i\geq 0$)
see e.g. \cite{Cheliotis} and \cite{Cheliotis_Favorite} (resp.  \cite{AndreolettiDevulder}). For some applications
to recurrent RWRE, see e.g. \cite{Bovier_Faggionato} and \cite{Devulder_Persistence}.

Let $C_1>2$ and $\alpha>2$.
Define $\log^{(2)} x=\log \log x$ for $x>1$.
As in \cite{Devulder_Persistence}, we use the
Koml\'os-Major-Tusn\'ady almost sure invariance principle \cite{KMT},
which ensures that:
\begin{lem}\label{lem:KMT}
Up to an enlargement
of $(\Omega,\mathcal F,\p)$, there exist
two independent two-sided Brownian motions $\big(W^{(i)}(s), s\in\R\big)$ ($i\in\{1,2\}$)
with $\E\big[(W^{(i)}(1))^2\big]=\E[(V^{(i)}(1))^2]=\sigma_i^2$ and a real number $\tilde C_1>0$ such that for all $n$ large enough,
$$
    \p\left[\sup_{|t|\le(\log n)^\alpha}\Big|V^{(i)}(\lfloor t\rfloor)-W^{(i)}(t)\Big|>\tilde C_1\log^{(2)} n\right]
\le
    (\log n)^{-C_1},
\qquad
    i\in\{1,2\}.
$$
\end{lem}
\begin{proof}
Notice that $V^{(1)}$ and $V^{(2)}$ are independent, since $\omega^{(1)}$ and $\omega^{(2)}$ are independent.
Due to (\cite{KMT}, Thm. 1),
there exist positive constants $a$, $b$ and $c$
such that for $N\in\mathbb N$,
up to an enlargement of $(\Omega,\mathcal F,\p)$,
there exist two independent two-sided Brownian motions $(W^{(i)}(s), s\in \R)$ ($i\in\{1,2\}$) on $(\Omega,\mathcal F,\p)$
with $\E[(W^{(i)}(1))^2]=\E[(V^{(i)}(1))^2]=\sigma_i^2$ such that
\begin{equation}\label{KMT}
%    \forall N\in\mathbb N,\quad\forall x\in\mathbb R,\quad
    \forall  x\in\mathbb R,\ \forall i\in\{1,2\},
\qquad
    \p\left[\sup_{|k|\le N}\Big|V^{(i)}(k)-W^{(i)}(k)\Big|>a\log N+x\right]
\le
    b e^{-cx}.
\end{equation}
Applying this result to $N:=\lfloor (\log n)^\alpha\rfloor+1$ and $x:=(\log (2b)+C_1\log^{(2)} n)/c$
and taking $\tilde C_1>2\left(a\alpha+\frac {C_1}c\right)$, we obtain that
\begin{equation}\label{eqKMT1}
    \p\left[\sup_{|k|\le\lfloor (\log n)^\alpha\rfloor+1}\Big|V^{(i)}(k)-W^{(i)}(k)\Big|
                >\frac{\tilde C_1}2\log^{(2)} n
      \right]
\le
    \frac 12(\log n)^{-C_1},
\end{equation}
for all $n$ large enough.
Moreover, for every $n$ large enough,
\begin{multline*}
    \ \p\left[\sup_{|t|\le(\log n)^\alpha}\Big|W^{(i)}(\lfloor t\rfloor)-W^{(i)}(t)\Big|>\frac{\tilde C_1}2\log^{(2)}  n\right]
\\
\le
    3(\log n)^\alpha\p\left[\sup_{0\le t<1}\big|W^{(i)}(t)\big|>\frac{\tilde C_1}2\log^{(2)} n\right]
\le
    6(\log n)^\alpha\p\left[\big|W^{(i)}(1)\big|>\frac{\tilde C_1}2\log^{(2)} n\right]
\\
\le
    6(\log n)^\alpha\frac{2}{\sqrt{2\pi}}e^{-\frac{(\tilde C_1)^2}{8\sigma_i^2}(\log^{(2)} n)^2}
=
    \frac{12}{\sqrt{2\pi}}(\log n)^{\alpha -\frac{(\tilde C_1)^2}{8\sigma_i^2}\log^{(2)} n}
\le
    \frac 12(\log n)^{-C_1}\, ,
\end{multline*}
since $\sup_{[0,1]} W^{(i)}=_{law} \big|W^{(i)}(1)\big|$.
This combined with \eqref{eqKMT1} proves the lemma.
\end{proof}
In the rest of the paper, we use the $W^{(i)}$ introduced in Lemma 4.2.
We will use the valleys for the $W^{(i)}$.
Fix some $C_2\geq 2\alpha+2+10\,\tilde C_1$. Let
\begin{equation}\label{hndef}
h_n:= \log n -5C_2\log^{(2)} n\, .
\end{equation}
We know from (\cite{Cheliotis}, Lemma 8) that $W^{(i)}\in\mathcal{W}$ almost surely (recall definition \ref{calW}).
%Moreover, applying \cite[Th 2.1]{HuShi} several times
Moreover, using \cite[Th 2.1]{HuShi} with $0<a=b$, we have
$
    \p\big[\sup_{0\leq s \leq t}[W^{(i)}(s)-\underline{W}^{(i)}(s)]<b\big]
\leq
    (4/\pi)\exp[-\pi^2\sigma_i^2 t/(8b^2)]
$,
where $\underline{W}^{(i)}(s):=\inf_{[0,s]} W^{(i)}$.
Applying this several times to $W^{(i)}$ and $-W^{(i)}$ with $t=(\log n)^\alpha/10$ and $b=h_n$,
%(see also \cite[last page]{HuShiModerate}),
the following holds
with a  probability $1-o\big((\log n)^{-2}\big)$ (since $\alpha>2$),
\begin{equation}\label{InegHypothesesMiPetits}
    \forall i\in\{1,2\},\quad    -(\log n)^{\alpha}
\leq
    b_{-4}\big(W^{(i)}, h_n\big)
\leq
    M_{3}\big(W^{(i)}, h_n\big)
\leq
    (\log n)^{\alpha}.
\end{equation}

%??????????????????????Link with localisation of diffusions and Sinai walk (explain, references).??????????????,,

The following lemma shows that Proposition \ref{Lemmeannexe2} is more subtle than it may seem at first sight.
\begin{lem}\label{LemmaCardinalHTiPetits}
Let $W$ be a two-sided standard Brownian motion and $\sigma >0$. Then, for every $n$ large enough,
\begin{equation}
\label{InegProbaHT0}
    \p\big[H(T_0(\sigma W,h_n))\leq \log n\big]
\geq
    C_2(\log^{(2)} n)(\log n)^{-1},
%    \frac{C_2}{\sigma}(\log^{(2)} n)(\log n)^{-1},
\end{equation}
\begin{eqnarray}
\label{InegProbaHCardTi}
&&
    \p\big[
        \sharp\{j\in\{-5,...,5\},\ H(T_j(\sigma W,h_n-2C_2\log^{(2)} n))\leq \log n + C_2\log^{(2)} n\}\geq 2
    \big]
\\
& = &
    O\big((\log^{(2)} n)^2(\log n)^{-2}\big),
\nonumber
\\[2mm]
\label{InegProbaHCardTi2}
&&
    \p\big[\exists j\in\{-5,...,5\},\ H(T_j(\sigma W,h_n-2C_2\log^{(2)} n))\le \log n + C_2\log^{(2)} n
    \big]
\\
& = &
    O\big((\log^{(2)} n)(\log n)^{-1}\big).
\nonumber
\end{eqnarray}
\end{lem}

In particular, the probability that the height of the central valley for $W^{(i)}$ is less than $\log n$ is not negligible.
However, with large enough probability, all the valleys close to $0$ except maybe one are large, with height
greater than $\log n + C_2\log^{(2)} n$.

\begin{proof}[Proof of Lemma \ref{LemmaCardinalHTiPetits}]
Let $\widetilde h_n:=h_n-2C_2\log^{(2)} n$.
First, due to (\cite{NP}, Prop. 1, see also \cite{Cheliotis} eq. (8)),
$e\big(T_i\big(\sigma W,\widetilde h_n\big)\big)/\widetilde h_n$ is for $i\neq 0$ an exponential random variable with mean $1$.
Consequently, for $i\neq 0$ and large $n$,
$$
    \p\big[H\big(T_i\big(\sigma W, \widetilde h_n\big)\big)\leq \log n+C_2\log^{(2)} n \big]
=
    \p\big[e\big(T_i\big(\sigma W, \widetilde h_n\big)\big)\leq 8C_2\log^{(2)} n\big]
%\leq
%    \frac{6C_2\log_2 n}{h_n}
\leq
    \frac{9C_2\log^{(2)} n}{\log n}.
$$
Observe that $e\big(T_0\big(\sigma W,\widetilde h_n\big)\big)/\widetilde h_n$
is by scaling equal in law to $e(T_0(W,1))$, which has
%is by scaling equal in law to $e(T_0(\sigma W,1))$, which has,
 a density equal to
$
%    (2\sigma^{-1}x+1)e^{-x/\sigma}{\mathbf 1}_{(0,\infty)}(x)/(3\sigma)
    (2x+1)e^{-x}{\mathbf 1}_{(0,\infty)}(x)/3
$
due to (\cite{Cheliotis}, formula (11)).
Hence for large $n$,
\begin{eqnarray*}
%    \frac{C_2\log_2 n}{\sigma\log n}
%    \frac{C_2\log_2 n}{\log n}
    C_2(\log^{(2)} n)(\log n)^{-1}
& \leq &
    \p\big[e\big(T_0\big(\sigma W,\widetilde h_n\big)\big)\leq 5C_2\log^{(2)} n\big]
\\
& \leq &
    \p\big[e\big(T_0\big(\sigma W,\widetilde h_n\big)\big)\leq 8C_2\log^{(2)} n\big]
\leq
    9C_2(\log^{(2)} n)(\log n)^{-1}.
%    \frac{9C_2\log_2 n}{\log n}.
%    \frac{7C_2\log_2 n}{\sigma\log n}.
\end{eqnarray*}
This remains true if $\widetilde h_n$ is replaced by $h_n$.
These inequalities already prove \eqref{InegProbaHT0} and \eqref{InegProbaHCardTi2}.
Moreover, due to (\cite{NP}, Prop. 1), the slopes $T_i\big(\sigma W, \widetilde h_n\big)$, $i\in\Z$ are independent,
up to their sign, so the
random variables $H\big(T_i(\sigma W,\widetilde h_n)\big)$, $i\in\Z$ are independent.
This and the previous inequalities lead to \eqref{InegProbaHCardTi}.
\end{proof}

Because of \eqref{InegProbaHT0}, it seems reasonable to consider strictly more than one valley of height at least $h_n$ if we
want to localize a recurrent RWRE with probability $\geq 1-(\log n)^{-2+\e}$ for $\varepsilon>0$.
%So one first idea would be to prove:
%\begin{conj} Let $\e>0$. For $\alpha\in(1/2,1)$, for large $n$,
%\begin{equation}\label{isittrue}
%    \mathbb P \left[\exists j\in\{-1,0,1\},\ |Z_n^{(i)}-b_j(W^{(i)},h_n)|\leq (%\log n)^\alpha \right]
%\geq
%    1-(\log n)^{-(1+\e)}.
%\end{equation}
%\end{conj}
%We do not know if \eqref{isittrue} holds true. However we can use, instead,
%Proposition \ref{LemmaLocalisationVerticale} below,
%which is less precise but seems easier to prove.

We first introduce some notation.
Let, for $i \in \{1,2\}$, $j\in\Z$ and $n\geq 3$,
\begin{align*}
&
    \Xi_{n,j}\big(W^{(i)}\big)
\\
& :=
    \big\{
        x\in\big[M_{j-1}\big(W^{(i)},h_n\big), M_j\big(W^{(i)},h_n\big)\big],
        \ W^{(i)}(x)\leq W^{(i)}\big(b_j\big(W^{(i)},h_n\big)\big)+C_2\log^{(2)} n
    \big\}.
\end{align*}
Loosely speaking, $\Xi_{n,j}\big(W^{(i)}\big)$ is the set of points with low potential in the $j$-th valley for $W^{(i)}$.
We also define
\begin{eqnarray*}
    \Xi_n\big(W^{(i)}\big)
& := &
    \bigcup_{j=-2}^2
        \Xi_{n,j}\big(W^{(i)}\big).
%        \Xi_{n,-2}\big(W^{(i)}\big)
%    \cup
%        \Xi_{n,-1}\big(W^{(i)}\big)
%    \cup
%        \Xi_{n,0}\big(W^{(i)}\big)
%    \cup
%        \Xi_{n,1}\big(W^{(i)}\big)
%    \cup
%        \Xi_{n,2}\big(W^{(i)}\big).
\end{eqnarray*}
In Proposition \ref{LemmaLocalisationVerticale} (proved in Section \ref{proofofLocVert}), we localize the RWRE $Z^{(i)}$ in a set of points
which are close to the $b_j(.)$ "vertically",
instead of "horizontally" as in Sinai's theorem (see \cite{S82}).

\begin{prop}\label{LemmaLocalisationVerticale}
Let $\varepsilon>0$
and $i\in\{1,2\}$.
For all $n$ large enough, we have
$$
    \mathbb P \big[Z_n^{(i)}\notin \Xi_n\big(W^{(i)}\big)\big]
\leq
    q_n:=(\log n)^{-2+\e}.
$$
\end{prop}
Proposition \ref{Lemmeannexe2} is then an easy consequence of Proposition \ref{LemmaLocalisationVerticale} and of the following estimate on the environments.
\begin{lem}\label{LemmaIntersectionBottoms2IndependentPotentials}
Let $\e>0$. For large $n$,
$$
    \p\big[ \Xi_n\big(W^{(1)}\big) \cap \Xi_n\big(W^{(2)}\big) \neq \emptyset \big]
\leq
    (\log n)^{-2+\e}.
$$
\end{lem}

%The two previous lemmas could be generalised to more than two Sinai walks.

\begin{proof}[Proof of Lemma \ref{LemmaIntersectionBottoms2IndependentPotentials}]
First, let $k\in \Xi_n\big(W^{(i)}\big)$ for some $i\in\{1,2\}$ and $n \geq 3$.
Hence
$k\in\big[M_{j-1}\big(W^{(i)},h_n\big), M_j\big(W^{(i)},h_n\big)\big]$
and
$W^{(i)}(k)\leq W^{(i)}\big(b_j\big(W^{(i)}, h_n\big)\big)+C_2\log^{(2)} n$
for some $j\in\{-2,-1,0,1,2\}$.
By definition of $h_n$-minima, we notice that the two Brownian motions
$\big(W^{(i)}(x+k)-W^{(i)}(k),\ x\geq 0\big)$
and
$\big(W^{(i)}(-x+k)-W^{(i)}(k),\ x\geq 0\big)$
hit $h_n-C_2\log^{(2)} n$ before $-2C_2\log^{(2)} n$.
By independence, it follows that, for $n$ large enough, for every $k\in\Z$ and $i\in\{1,2\}$,
\begin{eqnarray*}
    \p\big[k\in\Xi_n\big(W^{(i)}\big)\big]
& \leq &
    \p\big[T_{(i)}^+(h_n-C_2\log^{(2)} n)<T_{(i)}^+(-2C_2\log^{(2)} n)\big]
\nonumber\\
&&
    \qquad\times\p\big[\big|T_{(i)}^-(h_n-C_2\log^{(2)} n)\big|<\big|T_{(i)}^-(-2C_2\log^{(2)} n)\big|\big]
\nonumber\\
& \leq &O\left(((\log^{(2)} n)/\log n)^2\right),
\end{eqnarray*}
where
$T_{(i)}^+(z):=\inf\{x>0 \ :\ W^{(i)}(x)=z\}$
and
$T_{(i)}^-(z):=\sup\{x<0\ :\ W^{(i)}(x)=z\}$.
Consequently, since $W^{(1)}$ and $W^{(2)}$ are independent, we have uniformly on $k\in\Z$,
\begin{eqnarray}
    \p\big[k\in \Xi_n\big(W^{(1)}\big) \cap \Xi_n\big(W^{(2)}\big)\big]
& = &
    \p\big[k\in \Xi_n\big(W^{(1)}\big)\big]
    \p\big[k\in\Xi_n\big(W^{(2)}\big)\big]
\nonumber\\
& \leq &
    O\big((\log^{(2)} n)^4/(\log n)^4\big).
\label{eqProbaApparteniraXin}
\end{eqnarray}

Finally, \eqref{InegHypothesesMiPetits} applied with $2+\e>2$ instead of $\alpha$ and \eqref{eqProbaApparteniraXin} lead to
\begin{eqnarray*}
&&
    \p\big[ \Xi_n\big(W^{(1)}\big) \cap \Xi_n\big(W^{(2)}\big) \neq \emptyset \big]
\\
& \leq &
    o\big((\log n)^{-2}\big)
    +
    \sum_{k=-\lfloor(\log n)^{2+\e}\rfloor}^{\lfloor (\log n)^{2+\e}\rfloor}
    \p\big[k\in \Xi_n\big(W^{(1)}\big) \cap \Xi_n\big(W^{(2)}\big)\big]
\\
& \leq &
   O\left((\log^{(2)} n)^4(\log n)^{-2+\e}\right)
\leq
    (\log n)^{-2+2\e},
\end{eqnarray*}
for every $n$ large enough.
Since this is true for every $\e>0$, this proves the lemma.
\end{proof}

\begin{proof}[Proof of Proposition \ref{Lemmeannexe2}]
We have for large $n$, due to Proposition \ref{LemmaLocalisationVerticale},
\begin{eqnarray*}
    \mathbb P \big[Z_n^{(1)}=Z_n^{(2)}\big]
& \leq &
    \mathbb P \big[Z_n^{(1)}=Z_n^{(2)}, Z_n^{(1)}\in \Xi_n\big(W^{(1)}\big), Z_n^{(2)}\in \Xi_n\big(W^{(2)}\big)\big]
\\
&&
    \qquad
    +\mathbb P \big[ Z_n^{(1)}\notin \Xi_n\big(W^{(1)}\big)\big]
    +\mathbb P \big[ Z_n^{(2)}\notin \Xi_n\big(W^{(2)}\big)\big]
\\
& \leq &
    \p\big[ \Xi_n\big(W^{(1)}\big) \cap \Xi_n\big(W^{(2)}\big) \neq \emptyset \big]
    +2q_n.
\end{eqnarray*}
This and Lemma \ref{LemmaIntersectionBottoms2IndependentPotentials} prove
 	Proposition \ref{Lemmeannexe2}.
\end{proof}
\subsection{Proof of Proposition \ref{LemmaLocalisationVerticale}}
\label{proofofLocVert}
We fix $i\in\{1,2\}$.
To simplify notations we write $V$ for $V^{(i)}$,  $Z_n$ for $Z_n^{(i)}$ and $W$
for $W^{(i)}$.

The difficulty of this proof is that we have to localize $Z_n$ with probability $1-(\log n)^{-2+\e}$ instead of $1-o(1)$ as Sinai did in \cite{S82}. For this reason we need to take into account some cases which are usually considered to be negligible.
In order to prove Proposition \ref{LemmaLocalisationVerticale}, we first build a set  $\mathcal{G}_n$ of good environments, having high probability.
We prove that on such a good environment, the RWRE $Z=(Z_n)_n$ will reach quickly the bottom $b_{\mathcal I_1}$ of one of the two valleys of $W$ surrounding $0$. We need to consider these two valleys because we cannot neglect the case in which $0$ is close to the maximum of the potential between these two valleys.
\\
\\
Also, we cannot exclude that the valley surrounding $b_{\mathcal I_1}$ is "small",
that is, its height is close to $\log n$.
Then, we have to consider two situations. If the height of this valley is quite larger than $\log n$, then with large probability, $Z$ stays in this valley up to time $n$ (see Lemma \ref{lemE2}).
Otherwise (in the most difficult case, Lemma \ref{lemE2compl}), $Z$ can escape the valley surrounding  $b_{\mathcal I_1}$ before time $n$,
and in this case, with large probability, it reaches before time $n$ the bottom $b_{\mathcal I_2}$ of a neighbouring valley and stays in this valley up to time $n$.
In both situations, we prove that $Z_n$ is localized in $\Xi_n(W)$, and more precisely in the deepest places of the last valley visited before time $n$.
In order to prove this localization, we use the invariant measure of a RWRE in our environment, started at $b_{\mathcal I_1}$ or $b_{\mathcal I_2  }$.

We fix $\varepsilon>0$.
Recall \eqref{hndef}.
%, $x_j,b_j,M_j$ for respectively $x_j(W,h_n)$, $b_j(W,h_n)$ and $M_j(W,h_n)$.
We  introduce for $j\in\Z$,
$$
    x_j
:=
    \lfloor x_j(W,h_n)\rfloor,
\qquad
    b_j
:=
    \lfloor b_j(W,h_n)\rfloor,
\qquad
    M_j
:=
    \lfloor M_j(W,h_n)\rfloor.
$$
\begin{figure}[htbp]
\includegraphics[scale=0.92]{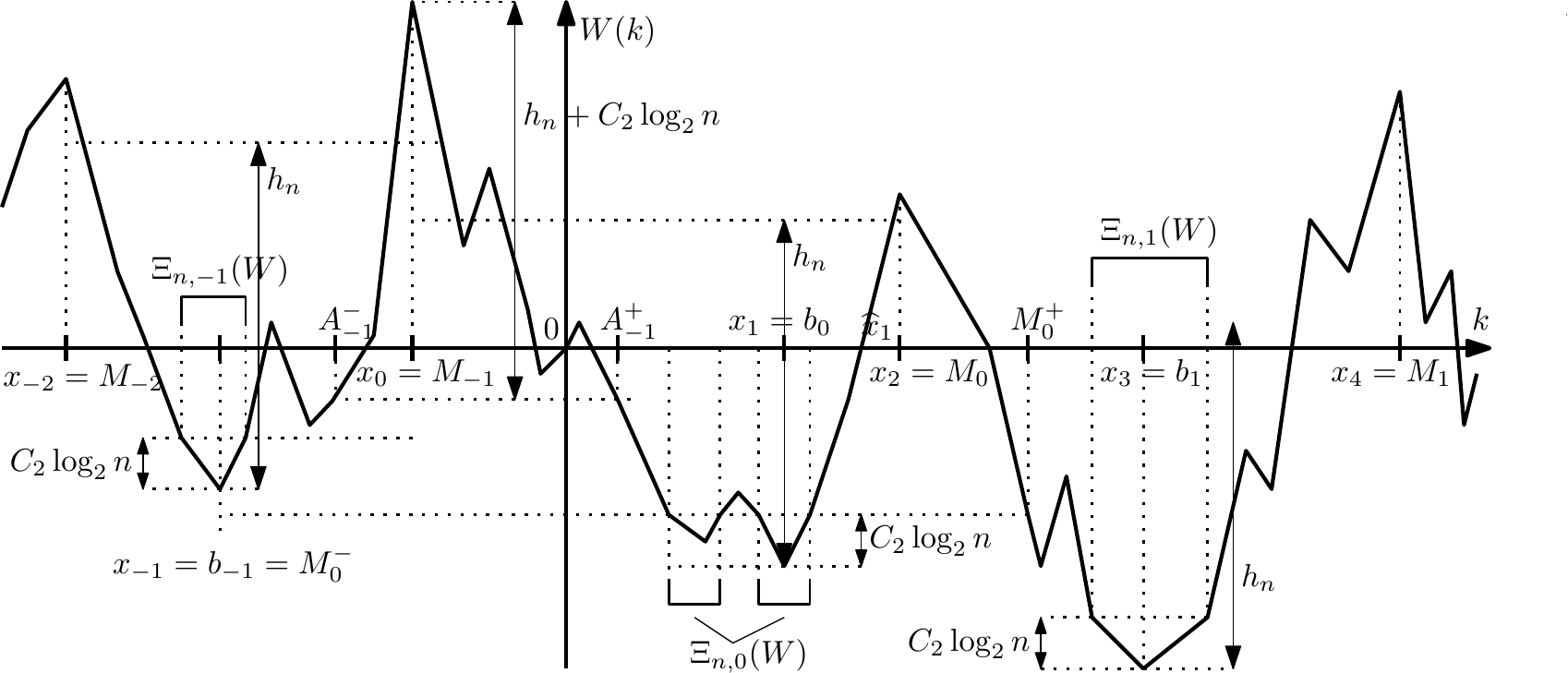}
\caption{Pattern of $W$ for a good environment $\omega\in\mathcal{G}_n$ and representation of different quantities.
}
\label{figure2_Good_Env}
\end{figure}

\noindent
We denote by $\mathcal{G}_n$ the set of {\bf good environments} $\omega$
satisfying \eqref{InegHypothesesMiPetits} together with the
following properties (see Figure \ref{figure2_Good_Env}):
\begin{equation}\label{KMT1}
\sup_{|t|\le(\log n)^\alpha}\big|V(\lfloor t\rfloor)-W(t)\big|\leq \tilde C_1\log^{(2)} n,
\end{equation}
\begin{equation}\label{eqDefGn1}
   \sharp\{j\in\{-5,..., 5 \},\ H(T_j(W,h_n-2C_2\log^{(2)} n))\leq \log n+C_2\log^{(2)} n\}\leq 1,
\end{equation}
with $h_n$ defined in \eqref{hndef}.
For every $n$ large enough, we have
\begin{equation}\label{PGn}
    \p\big[(\mathcal{G}_n)^c\big]
\leq
    (\log n)^{-2+\e}
,
\end{equation}
due to  \eqref{InegHypothesesMiPetits} and to Lemmas \ref{lem:KMT} and \ref{LemmaCardinalHTiPetits}, since $C_1>2$ .

We now consider
%$\omega\in\mathcal{G}_n$. We also suppose that
$\omega\in\mathcal{G}_n^+$ where $\mathcal{G}_n^+:=\mathcal{G}_n\cap\{x_1(W,h_n) \text{ is an $h_n$-minimum}\}$, that is,
$$
    b_{-1}(W,h_n)
=
    x_{-1}(W,h_n)
<
    x_0(W,h_n)
=
    M_{-1}(W,h_n)
\leq
    0
<
    b_0(W,h_n)
=
    x_1(W,h_n).
$$
Indeed, the other case, that is, $x_0(W,h_n)$ is an $h_n$-minimum, or equivalently
$\omega \in\mathcal{G}_n^-$ with $\mathcal{G}_n^-:=\mathcal{G}_n \backslash \mathcal{G}_n^+$,
is similar by symmetry.

\begin{proof}[Proof of Proposition \ref{LemmaLocalisationVerticale}]
Let us see how we can derive Proposition \ref{LemmaLocalisationVerticale}
from \eqref{PGn} and from Lemmas  \ref{lemAtteinteb_j}, \ref{lemE2} and \ref{lemE2compl} below.
Applying Lemma \ref{lemAtteinteb_j} with $y=0$ and $j=-1$ on $\mathcal{G}_n^+$,
the random walk $Z$ goes quickly to $b_{-1}$ or $b_0$ with high probability.
More precisely, setting $E_1:=\{\tau(b_{-1})\wedge\tau(b_0)\leq n(\log n)^{-3C_2}\}$,
there exists $\tilde n_0\in\N$ such that, for every $n\ge\tilde n_0$,
\begin{equation}
\label{controlE1a}
    \forall\omega\in\mathcal G_n^+,
\qquad
    P_\omega(E_1)\geq 1-(\log n)^{-2}.
\end{equation}
Due to Lemmas \ref{lemE2} and \ref{lemE2compl}, there exists $\tilde n_1\in\N$ such that, for every
$n\ge \tilde n_1$,
\begin{equation*}
%\label{InegProbaZDansGoodEnvironment}
    \forall \omega\in\mathcal{G}_n^+,
\qquad
    P_{\omega} \big[E_1,\ Z_n\notin \Xi_n(W)\big]
\leq
    11(\log n)^{-2}
\end{equation*}
and so, using \eqref{controlE1a},
$$
    \forall n\ge\max(\tilde n_0,\tilde n_1),
\quad
    \forall \omega\in\mathcal{G}_n^+,
\qquad
    P_{\omega} \big[Z_n\notin \Xi_n(W)\big]
\leq
    12(\log n)^{-2}.
$$
By symmetry, this remains true with $\mathcal{G}_n^+$ replaced by $\mathcal{G}_n^-$.
Therefore, due to \eqref{PGn}, for every $n$ large enough,
$$
    \mathbb P \big[Z_n\notin \Xi_n(W)\big]
\leq
    \int_{\mathcal{G}_n}
    P_{\omega} \big[Z_n\notin \Xi_n(W)\big]
    \p(\text{d}\omega)
    +
    \p\big[(\mathcal{G}_n)^c\big]
\leq
    2(\log n)^{-2+\e}.
$$
Since this is true for every $\e>0$, this proves Proposition \ref{LemmaLocalisationVerticale}.
\end{proof}

We will use the following property. For $j\in\Z$, let
\begin{eqnarray*}
    \widehat \mu_j(x)
& := &
    \exp\big[-\big(V(x)-V(b_j)\big)\big]
    +
    \exp\big[-\big(V(x-1)-V(b_j)\big)\big]
\\
& = &
    \exp\big[V(b_j)\big]
    \mu_\omega(x),
\qquad
    x\in\Z,
\end{eqnarray*}
with reversible measure $\mu_\omega$ defined in \eqref{reversiblemeas}.
It follows from reversibility that
\begin{equation}
\label{InegReversible}
    \forall k\in\N,\,
    \forall x\in\Z,\,
    \forall b\in\Z,\,
\quad
P_{\omega}^b[Z_k=x]
=
    \frac{\mu_\omega(x)}{\mu_\omega(b)} P_{\omega}^x[Z_k=b] \leq \frac{\mu_\omega(x)}{\mu_\omega(b)}
\leq
    \exp[V(b)]\mu_\omega(x)\, .
\end{equation}
In particular,
\begin{equation}\label{InegInvariantMeasureInduction}
    \forall j\in\Z,\
    \forall k\in\N,\
    \forall x\in\Z,
\qquad
    P_{\omega}^{b_j}[Z_k=x]
\leq
    \widehat \mu_j(x).
\end{equation}

\medskip

\begin{lem}\label{lemAtteinteb_j}
There exists $n_0\in\N$ such that, for every $n\ge n_0$, every $\omega\in\mathcal G_n$,
every $j\in\{-2,...,1\}$ and every integer
$y \in]b_j, b_{j+1}[$,
%$y \in]b_j(W,h_n), b_{j+1}(W,h_n)[$,
\begin{equation}\label{eqLemmaSortieVallee}
    P_\omega^y
    \big[
        \tau\left(b_j\right)\wedge\tau\left(b_{j+1}\right)
        %\tau\left(b_j(W,h_n)\right)\wedge\tau\left(b_{j+1}(W,h_n)\right)
        >
        n(\log n)^{-3C_2}
    \big]
\leq
    (\log n)^{-2}.
\end{equation}
\end{lem}

%\smallskip

\begin{proof}
Let $j\in\{-2,...,1\}$ and $\omega\in\mathcal G_n$.
Assume for example that $y\in[M_j,b_{j+1}[$, the proof being symmetric in the case when $y\in]b_j,M_j]$.
We set (see Figure \ref{figure2_Good_Env} for $j=-1$)
\begin{eqnarray*}
    A_j^+
& := &
    \min(b_{j+1},\inf\{k\geq M_j\ : \ W(k)\le W(M_j(W,h_n))-h_n-C_2\log^{(2)} n\}),
\\
    A_j^-
& := &
    \max(b_{j},\sup\{k\leq M_j : \ W(k)\le W(M_j(W,h_n))-h_n-C_2\log^{(2)} n\}).
\end{eqnarray*}
%\begin{itemize}
1.
If $y\in[M_j,A_j^+[$,
%If $y\in[M_j,A_j^+[$,
due to \eqref{InegEsperance1}, \eqref{InegHypothesesMiPetits} and \eqref{KMT1},
applying Markov's inequality, we get
%\begin{eqnarray*}
\begin{multline*}
    P_\omega^y\left[\tau(A_j^-)\wedge\tau(A_j^+)>\frac{e^{h_n+2C_2\log^{(2)} n}}{2}\right]
\leq
    \frac{2\e_0^{-1}[b_{j+1}-b_{j}]^2}{e^{h_n+2C_2\log^{(2)} n}}\exp\Big[\max_{[b_j,b_{j+1}]} V-\min_{[A_j^-, A_j^+]} V\Big]
\\
\leq
    \frac{2\e_0^{-1}[M_2-b_{-2}]^2}{e^{h_n+2C_2\log^{(2)} n}}
    e^{\big(W(M_j(W,h_n))+\tilde C_1\log^{(2)} n\big)-\big(W(M_j(W,h_n))-h_n-C_2\log^{(2)} n-\tilde C_1\log^{(2)} n-\log\e_0^{-1}\big)}
%    \exp\Big[\Big(V(M_j)+2\tilde C_1\log^{(2)} n\Big)-\Big(V(M_j)-h_n-C_2\log^{(2)} n-\log\frac{1-\e_0}{\e_0}\Big)\Big]
\\
\le
    \frac{8\varepsilon_0^{-1}(\log n)^{2\alpha}e^{h_n+(C_2+2\tilde C_1)\log^{(2)} n+\log \e_0^{-1}}}{e^{h_n+2C_2\log^{(2)} n}}
\le
    8\varepsilon_0^{-2}(\log n)^{2\alpha+2\tilde C_1-C_2}
\le
    \frac 13(\log n)^{-2}
%\end{eqnarray*}
\end{multline*}
for every $n$ large enough, where we used
$\sup_{[b_j(W,h_n),b_{j+1}(W,h_n)]}W=W(M_j(W,h_n))$ and
$V(A_j^{\pm})\geq V(A_j^{\pm}\mp 1)-\log\frac{1-\e_0}{\e_0}$ in the second line
and $C_2>2\alpha+2\tilde C_1+2$
%and $\varepsilon<1$
in the last one.
%We prove similarly the same inequality if $y\in]A_j^-, M_j]$.
Hence by the strong Markov property, for $n$ large enough, for every $y\in[M_j,A_j^+]$,
\begin{multline}
\label{eqPobataubjprocheMj}
    P_\omega^y\left[\tau(b_j)\wedge\tau(b_{j+1})>e^{h_n+2C_2\log^{(2)} n}\right]
\\
\le
    \frac 13(\log n)^{-2}+ P^{A_j^-}_\omega\left[\tau(b_j)>e^{h_n+2C_2\log^{(2)} n}/2\right]
    +
    P^{A_j^+}_\omega\left[\tau(b_{j+1})>e^{h_n+2C_2\log^{(2)} n}/2\right].
\end{multline}
2. Assume now that $y\in[A_j^+,b_{j+1}[$ (and so $A_j^+<b_{j+1}$).
Observe that $W$ admits no $h_n$-maximum in the interval $]M_j(W,h_n),b_{j+1}(W,h_n)]$ by definition of $M_j(.)$,
so
$$\max_{M_j(W,h_n)\leq u\leq v\leq b_{j+1}(W,h_n)}(W(v)-W(u))<h_n.$$
Hence due to \eqref{InegEsperance1}, \eqref{InegHypothesesMiPetits}, \eqref{KMT1}, and to Markov's inequality, we have
\begin{align}
    P^{y}_\omega\left[\tau(M_j)\wedge\tau(b_{j+1})>e^{h_n+2C_2\log^{(2)} n}/2\right]
%\leq
%    \frac{2}{e^{h_n+2C_2\log_2 n}}
%    E^{y}_\omega\left[\tau(M_j)\wedge\tau(b_{j+1})\right]
&\le
    \frac{2[M_2-b_{-2}]^2}{\e_0 e^{h_n+2C_2\log^{(2)} n}}
    \exp\Big[\max_{M_j\leq \ell\leq k\leq b_{j+1}}\big(V(k)-V(\ell)\big)\Big]
\nonumber\\
&\le
    \frac{8\varepsilon_0^{-1}(\log n)^{2\alpha} e^{h_n+2\tilde C_1\log^{(2)} n}}{e^{h_n+2C_2\log^{(2)} n}}
\le
    \frac 16(\log n)^{-2}
\label{pinf}
\end{align}
for every $n$ large enough, since $2C_2>2\alpha+2\tilde C_1+2$.
Moreover, due to \eqref{probaatteinte}, \eqref{InegHypothesesMiPetits} and \eqref{KMT1},
and since there is no $h_n$-maximum in $[A_j^+, b_{j+1}]$
and so $\sup_{[A_j^+, b_{j+1}]}W<W(A_j^+)+h_n$,
\begin{eqnarray}
    P^{y}_\omega\left[\tau(M_j)<\tau(b_{j+1})\right]
& \leq &
    \bigg(\sum_{\ell=A_j^+}^{b_{j+1}-1}e^{V(\ell)}\bigg)\bigg(\sum_{\ell=M_j}^{b_{j+1}-1}e^{V(\ell)}\bigg)^{-1}
    %\frac{\sum_{\ell=A_j^+}^{b_{j+1}-1}e^{V(\ell)}}{\sum_{\ell=M_j}^{b_{j+1}-1}e^{V(\ell)}}
\nonumber\\
&\le&
    \big[b_{j+1}-A_j^+\big] \exp\Big(\max\limits_{\ell\in\{A_j^+,...,b_{j+1}\}}V(\ell)\Big)\exp\big(-V(M_j)\big)
%\nonumber\\
%&\le&
%    \frac{2(\log n)^\alpha e^{\max_{\ell\in\{A_j^+,...,b_{j+1}\}}V(\ell)}}{e^{V(M_j)}}
\nonumber\\
&\le&
    2(\log n)^{\alpha}\exp\big[W(A_j^+)+h_n-W(M_j(W,h_n))+2\tilde C_1\log^{(2)} n\big]
\nonumber\\
&\le&
    2(\log n)^{\alpha-C_2+2\tilde C_1}
\le
    \frac 16(\log n)^{-2}\label{pmin}
\end{eqnarray}
for every $y\in[A_j^+,b_{j+1}[$
for all $n$ large enough, since $C_2>\alpha+2\tilde C_1+2$.
Gathering \eqref{pinf} and \eqref{pmin}, we get, for all $n$ large enough,
for every $y\in[A_j^+,b_{j+1}[$, uniformly on $\mathcal{G}_n$ as the previous inequalities,
\begin{equation}\label{eqProbaAtteintebj+1}
  P^{y}_\omega\left[\tau(b_{j+1})>n(\log n)^{-3C_2}/2\right]=   P^{y}_\omega\left[\tau(b_{j+1})>e^{h_n+2C_2\log^{(2)} n}/2\right]
\le
    \frac 13(\log n)^{-2},
\end{equation}
recalling \eqref{hndef}. This already proves \eqref{eqLemmaSortieVallee} for $y\in[A_j^+,b_{j+1}[$.

3. For symmetry reasons, we also get that, for every $n$ large enough,
for every $y\in]b_j,A_j^-]$,
\begin{equation}\label{eqProbaAtteintebj}
    P^{y}_\omega\left[\tau(b_{j})>e^{h_n+2C_2\log^{(2)} n}/2\right]\le \frac 13(\log n)^{-2}.
\end{equation}
Finally, combining \eqref{eqPobataubjprocheMj} with \eqref{eqProbaAtteintebj+1} and \eqref{eqProbaAtteintebj} proves
\eqref{eqLemmaSortieVallee} for $y\in[M_j, A_j^+[$. Hence,
\eqref{eqLemmaSortieVallee} is true for $y\in[M_j, b_{j+1}[$ thanks to 2.,
and for $y\in]b_j, M_j]$ by symmetry. This proves the lemma.
\end{proof}
We consider $\mathcal I_1\in\{-1,0\}$ such that $\tau\big(b_{\mathcal I_1}\big)=\tau(b_{-1})\wedge\tau(b_0)$.
Recall that
$
    E_1
=
    E_1(n)
:=
    \{
    \tau\big(b_{\mathcal I_1}\big)
    %=\tau(b_{-1})\wedge\tau(b_0)
    \leq \alpha_n\}
$,
where we set
\begin{equation}\label{alphandef}
    \alpha_n
:=
    n(\log n)^{-3C_2}.
\end{equation}
We already saw in \eqref{controlE1a} that, thanks to Lemma \ref{lemAtteinteb_j} with $y=0$ and $j=-1$, we have
\begin{equation*}%\label{controlE1}
    \forall\omega\in\mathcal G_n^+,
\qquad
    P_\omega(E_1)\geq 1-(\log n)^{-2}.
\end{equation*}
%there exists $\mathcal I_1\in\{-1,0\}$ such that
We consider the event $E_2=E_2(n)$ on which $Z$ first goes to the bottom
of a "deep" valley:
\begin{eqnarray*}
    E_2^+(j)
& := &
    \{W[M_{j}(W,h_n)]-W[b_{j}(W,h_n)] > \log n+C_2\log^{(2)} n\},
\qquad j\in\Z,
\\
    E_2^-(j)
& := &
    \{W[M_{j-1}(W,h_n)]-W[b_{j}(W,h_n)] > \log n+C_2\log^{(2)} n\},
\qquad j\in\Z,
\\
    E_2
& := &
    E_2^+(\mathcal I_1)\cap E_2^-(\mathcal I_1).
\end{eqnarray*}
Notice that this event depends on $\omega$ but also on the first
steps of $Z$ up to time $\tau(b_{I_1})$.
Similarly as in \eqref{InegProbaHCardTi2}, this event happens
with probability $1-O((\log n)^{-1}\log^{(2)} n)$, so we cannot neglect $E_2^c$.
We will treat separately the two events $E_2$ and $E_2^c$ (the study of $E_2^c$
being more complicated).
Before considering these two events, we state the following useful result.
We introduce for $j\in\Z$,
\begin{eqnarray}
    M_{j}^+
& := &
    b_{j+1}
    \wedge
    \inf\{k\geq M_{j},\ W(k)\leq W(b_{j}(W,h_n))+C_2\log^{(2)} n\},
\label{eqDefMj+}
\\
%    M_{j}^-
%& := &
%    \sup\{k\leq M_{j},\ W(k)\leq W(b_{j})+C_2\log^{(2)} n\},
%\nonumber\\
%    M_{j,-1}^+
%& := &
%    \inf\{k\geq M_{j-1},\ W(k)\leq W(b_{j})+C_2\log^{(2)} n\},
%\nonumber\\
    M_{j}^-
& := &
    b_{j-1}
    \vee
    \sup\{k\leq M_{j-1},\ W(k)\leq W(b_{j}(W,h_n))+C_2\log^{(2)} n\},
\label{eqDefMj-}
\end{eqnarray}
where $u\vee v:=\max(u,v)$, $(u,v)\in\R^2$, so that
\begin{equation}\label{eqEp15}
    \forall k\in ]M_j^-,M_{j-1}]\cup [M_j,M_j^+[,
\quad
    W(k)>W(b_{j}(W,h_n))+C_2\log^{(2)} n.
\end{equation}

\smallskip

\begin{lem}
For every $n$ large enough,
\begin{equation}
\label{eq41Amelioree}
    \forall \omega\in\mathcal{G}_n,\,
    \forall j\in\{-2,\dots,2\},
\quad
    \sup_{k\geq 0}
    P_\omega^{b_{j}}\big(Z_k\in[M_{j}^-, M_j^+]\setminus \Xi_{n}(W)   \big)
\leq
    (\log n)^{-2}.
\end{equation}
\end{lem}
\begin{proof}
We claim that due to \eqref{KMT1} and \eqref{InegHypothesesMiPetits},
$V(x)\geq W(b_j(W,h_n))+C_2\log^{(2)} n-\tilde C_1\log^{(2)} n-\log\e_0^{-1}$ for every
integer $x\in\big(\big[M_{j}^-, M_j^+\big]\setminus \Xi_{n,j}(W)\big)$ for $j\in\{-2,\dots, 2\}$. This follows from
the definition of $\Xi_{n,j}(W)$ if $x\in[M_{j-1}, M_j]$,
and
%\eqref{eqDefMj+}, \eqref{eqDefMj-}
from \eqref{eqEp15} and the fact that
$|V(y)-V(y-1)|\leq \log \frac{1-\e_0}{\e_0}$ otherwise.
%$x\in[M_j^-,M_j^+]\cup [M_{j,-1}^-,M_{j,-1}^+]$ for $j\in\{-1,0\}$.
So, due to \eqref{InegInvariantMeasureInduction}, \eqref{InegHypothesesMiPetits} and \eqref{KMT1},
for large $n$, for all $\omega\in\mathcal{G}_n$ and $j\in\{-2,\dots,2\}$,
%We can prove, similarly as we did in (41) (number can change):
\begin{eqnarray*}
&&
    \sup_{k\geq 0}
    P_\omega^{b_{j}}\big(Z_k\in[M_{j}^-, M_j^+]\setminus \Xi_{n}(W)   \big)
\nonumber\\
& \leq &
   \sum_{x=M^-_{j}}^{M_{ j}^+}
    1_{\Xi_{n}(W)^c}(x)\widehat \mu_j(x)
\leq
   \sum_{x=M^-_{j}}^{M_{j}^+}
    1_{\Xi_{n,j}(W)^c}(x)
    e^{V(b_{j})}\big[e^{-V(x)}
    +e^{-V(x)+\log\frac{1-\e_0}{\e_0}}\big]
\nonumber\\
& \leq &
    2(\log n)^\alpha \e_0^{-1} e^{\big(W(b_j(W,h_n))+\tilde C_1\log^{(2)} n\big)
        -\big(W(b_j(W,h_n))+(C_2-\tilde C_1)\log^{(2)} n-\log\e_0^{-1}\big)}
\nonumber\\
& = &
    2\e_0^{-2}(\log n)^{\alpha+2\tilde C_1-C_2}
\leq
    (\log n)^{-2},
\end{eqnarray*}
since $C_2\geq 2\alpha+2+10\,\tilde C_1$.
\end{proof}
In the next lemma, we consider the case where $Z$ goes quickly in a deep valley.
\begin{lem}[Simplest case]\label{lemE2}
%For every $\varepsilon>0$,
There exists $n_1\in\N$ such that for all $n\ge n_1$,
$$
    \forall\omega\in \mathcal G_n^+,
\qquad
    P_\omega(E_1,E_2,Z_n\notin\Xi_n(\omega))\le 3(\log n)^{-2}.$$
\end{lem}
\begin{proof}
Due to \eqref{InegProba1}, \eqref{InegHypothesesMiPetits} and
\eqref{KMT1}, we have for large $n$,
for all $\omega\in \mathcal G_n^+$ and all  $j\in\{-2,\dots,2\}$ uniformly on  $E_2^+(j)$,
\begin{multline}
\label{PPPP1}
    P_\omega^{b_{j}}\left[\tau(M_{j})<n\right]
\leq
    n e^{\min_{[b_{j}, M_{j}-1]}V-V(M_{j}-1)}
\leq
    n e^{V(b_{j})-V(M_{j}-1)}
\\
\leq
    n \exp\big[\big(W(b_{j}(W,h_n))+\tilde C_1\log^{(2)} n\big)
               -
               \big(W(M_{j}(W,h_n))-\tilde C_1\log^{(2)} n-\log\e_0^{-1}\big)
          \big]
\\
\le
    n e^{-(\log n+C_2\log^{(2)} n)+2\tilde C_1\log_ 2 n+\log \varepsilon_0^{-1}}
=
    \e_0^{-1}(\log n)^{2\tilde C_1-C_2}
\le
    (\log n)^{-2},
\end{multline}
since $C_2\geq 2\alpha+2+10\,\tilde C_1$.
Similarly, using \eqref{InegProba2} instead of \eqref{InegProba1}, we have for large $n$,
for all $\omega\in \mathcal G_n^+$ and all $j\in\{-2,\dots,2\}$, uniformly on $E_2^-(j)$,
\begin{equation}\label{PPPP2}
    P_\omega^{b_{ j}}\left[\tau(M_{j-1})<n\right]
\le
    (\log n)^{-2}.
\end{equation}
Let
$$
    \tau(x,y)
:=
    \inf\{k\geq 0,\ Z_{\tau(x)+k}=y\},
\qquad
    x\in\Z,\, y\in\Z.
$$
In particular, on $E_1\cap E_2\cap\big\{    \tau\big(b_{\mathcal I_1}, M_{\mathcal I_1-1}\big) \geq n\big\}
\cap\big\{\tau\big(b_{\mathcal I_1}, M_{\mathcal I_1}\big) \geq n\big\}$, recalling \eqref{alphandef},
$$
    \tau\big(b_{\mathcal I_1}\big)
\leq
    \alpha_n
\leq
    n
\leq
    \tau\big(b_{\mathcal I_1}\big)
    +\tau\big(b_{\mathcal I_1}, M_{\mathcal I_1-1}\big)
    \wedge
    \tau\big(b_{\mathcal I_1}, M_{\mathcal I_1}\big),
$$
and so $Z_n\in \big[M_{\mathcal I_1-1}, M_{\mathcal I_1}\big]\subset\big[M_{\mathcal I_1}^-, M_{\mathcal I_1}^+\big]$.
Applying  \eqref{PPPP1} and \eqref{PPPP2} combined with the strong Markov property at time $\tau(b_{\mathcal I_1})$,
and then \eqref{eq41Amelioree}, we get for large $n$, for every $\omega\in\mathcal{G}_n^+$,
\begin{eqnarray}
&&
    P_\omega[E_1,E_2, Z_n\notin \Xi_n(W)]
\nonumber\\
& \leq &
    P_\omega\big[E_1,E_2, Z_n\notin \Xi_n(W), \tau\big(b_{\mathcal I_1}, M_{\mathcal I_1-1}\big) \geq n,
                    \tau\big(b_{\mathcal I_1}, M_{\mathcal I_1}\big) \geq n\big]
\nonumber\\
&&
    +P_\omega\big[E_2, \tau\big(b_{\mathcal I_1}, M_{\mathcal I_1-1}\big) < n\big]
    +P_\omega\big[E_2, \tau\big(b_{\mathcal I_1}, M_{\mathcal I_1}\big) < n\big]
\nonumber\\
& \leq &
    E_\omega\big[1_{E_1}
    P_\omega^{b_{\mathcal I_1}}\big(Z_{n-k}\in [M_{\mathcal I_1-1}, M_{\mathcal I_1}]
                            \setminus\Xi_{n}(W)\big)_{|k=\tau(b_{\mathcal I_1})}\big]
    +2(\log n)^{-2}
\nonumber\\
& \leq &
    3(\log n)^{-2}.
\label{eqConclusionE2}
\end{eqnarray}
This proves the lemma.
\end{proof}
For the event $E_2^c$, we will use the following lemma, which is actually true for any Markov chain.
\begin{lem}\label{LemmaProbaHittingTime}
Let $a\ne b$. We have,
$$
    \forall k\in\N,
\qquad
    P_\omega^b[\tau(a)=k]
\leq
    P_\omega^b[\tau(a)<\tau(b)].
$$
\end{lem}
\begin{proof}
Let $k\in\N^*$. We have, by the Markov property,
\begin{eqnarray*}
    P_\omega^b[\tau(a)=k]
& = &
    \sum_{n=0}^k P_\omega^b\big[\tau(a)=k, Z_n=b, \forall n<\ell \leq k, Z_\ell\neq b\big]
\\
& \leq &
    \sum_{n=0}^k P_\omega^b\big[Z_n=b\big]P_\omega^b\big[\tau(a)=k-n,  \tau(a)<\tau(b)\big]
%\\
%& \leq &
%    \sum_{n=0}^k P_\omega^b[\tau(a)=k-n,  \tau(a)<\tau(b)]
\\
& \leq &
    P_\omega^b[\tau(a)\in[0,k],  \tau(a)<\tau(b)]
\\
& \leq &
    P_\omega^b[\tau(a)<\tau(b)],
\end{eqnarray*}
where we used $P_\omega^b[Z_n=b]\leq 1$ in the second inequality.
\end{proof}
%Due to \eqref{probatouscas}, we conclude
%\begin{equation}\label{probacassimple}
%P_\omega\left[E_1,\ Z_n\in[M_{\mathcal I_1,-1}^+,M_{\mathcal I_1}^-]\right]=1-O\left((\log n)^{-1-\e}\right),
%\end{equation}
%that is, with high probability, $Z_n$ is localized in a valley of $W$.
\begin{lem}[Most difficult case]\label{lemE2compl}
There exists $n'_1\in\N$ such that for all $n\ge n'_1$,
$$
    \forall \omega\in \mathcal G_n^+,
\quad
    P_\omega(E_1,E_2^c,Z_n\notin\Xi_n(\omega))\le 8 (\log n)^{-2}.$$
\end{lem}
\begin{proof}
An essential remark is that if we are on $E_2^c$ with $\omega\in\mathcal G_n^+$, then, due to \eqref{eqDefGn1}, either we are on $E_2^-(\mathcal I_1)\setminus E_2^+(\mathcal I_1)$ or on $E_2^+(\mathcal I_1)\setminus E_2^-(\mathcal I_1)$.
In the first case we set
$$
    \mathcal I_2:=\mathcal I_1+1,\
    A:=M_{\mathcal I_1}^+,\
    B:=M_{\mathcal I_1}(W,h_n)\mbox{ and }
    D:=M_{\mathcal I_1}^-.
$$
whereas in the second case we set
$$
    \mathcal I_2:=\mathcal I_1-1,\
    A:=M_{\mathcal I_1}^-,\
    B:=M_{\mathcal I_1-1}(W,h_n)\mbox{ and }
    D:=M_{\mathcal I_1}^+.
$$
Loosely speaking, with large probability, $b_{\mathcal I_2}$ is the bottom of the second valley reached by $Z$,
and $Z$ can reach it before time $n$ or not, so we have to consider both cases.

We introduce
$
    \tau'(A, b_{\mathcal I_2})
:=
    \inf\{k\geq 0,\ Z_{\tau(b_{\mathcal I_1}) +\tau(b_{\mathcal I_1},A)+k}=b_{\mathcal I_2}\}
$
and
\begin{eqnarray*}
    E_3
& := &
    \{\tau(b_{\mathcal I_1})+\tau(b_{\mathcal I_1},A)<n-2n(\log n)^{-6C_2}\}
    \cap\{\tau'(A, b_{\mathcal I_2}) \leq n(\log n)^{-6C_2}\},
\\
    E_4
& := &
    \{\tau(b_{\mathcal I_1})+\tau(b_{\mathcal I_1},A)\in[n-2n(\log n)^{-6C_2},n]\},
\\
    E_5
& := &
    \{\tau(b_{\mathcal I_1})+\tau(b_{\mathcal I_1},A)>n\}\cap \{\tau(b_{\mathcal I_1}, D)>n\},
\\
    E_6
& := &
    \{\tau(b_{\mathcal I_1}, D)\leq n\},
  \\
    E_7
& := &
    \{\tau'(A, b_{\mathcal I_2}) \geq n(\log n)^{-6C_2}\}.
\end{eqnarray*}
Notice that
\begin{equation}\label{eqE2Union}
    E_2^c
\subset
    E_3\cup E_4\cup E_5\cup E_6\cup E_7.
\end{equation}
\begin{itemize}
\item \underline{Control on $E_6$}.
First, $E_2^c\cap\{\mathcal I_2=\mathcal I_1+1\}\subset E_2^-(\mathcal I_1)$, so by \eqref{PPPP2}
and since $D=M_{\mathcal I_1}^-<M_{\mathcal I_1-1}<b_{\mathcal I_1}$ when $\mathcal I_2=\mathcal I_1+1$,
we have for large $n$ for every $\omega\in\mathcal{G}_n^+$,
\begin{equation}\label{eqProbaE2E6}
    P_\omega(E_2^c\cap\{\mathcal I_2=\mathcal I_1+1\}\cap E_6)
\leq
    E_\omega\big[1_{E_2^-(\mathcal I_1)}P_\omega^{b_{\mathcal I_1}}\big(\tau(M_{\mathcal I_1-1})<\tau(D)\leq n\big)\big]
\leq
    (\log n)^{-2}.
\end{equation}
The case $\mathcal I_1=\mathcal I_2-1$ follows similarly from \eqref{PPPP1}, and so
$$
    P_\omega(E_2^c\cap E_6)
\leq
    2(\log n)^{-2}
$$
for large $n$ for every $\omega\in\mathcal{G}_n^+$.
\item \underline{Control on $E_4$}. We start by proving that
for every $n$ large enough, for every $\omega\in\mathcal G_n^+$, uniformly on $E_2^c$,
\begin{equation}\label{eqpourE4}
    \forall x\in\N,
\qquad
%    P_{\omega}^{b_{\mathcal I_1}}\left[\tau(A)\in[n-2\alpha_n,n]\right]
    P_{\omega}^{b_{\mathcal I_1}}\left[\tau(A)\in[n-2n(\log n)^{-6C_2}-x,n-x]\right]
 \leq
    (\log n)^{-2}.
\end{equation}
Using Lemma \ref{LemmaProbaHittingTime} and then \eqref{probaatteinte}, we obtain on $E_2^-(\mathcal I_1)\setminus E_2^+(\mathcal I_1)$, since $b_{\mathcal I_1}<M_{\mathcal I_1}<A$,
\begin{eqnarray*}
   P_{\omega}^{b_{\mathcal I_1}}
    \big[\tau(A)=\ell\big]
& \leq &
    P_{\omega}^{b_{\mathcal I_1}}\big[\tau(A) < \tau(b_{\mathcal I_1})\big]
=
   \omega_{b_{\mathcal I_1}}P_{\omega}^{b_{\mathcal I_1}+1} \big[\tau(A) < \tau(b_{\mathcal I_1})\big]
\\
& \leq &
    e^{V(b_{\mathcal I_1})-V(M_{\mathcal I_1})}
\leq
    e^{W[b_{\mathcal I_1}(W,h_n)]- W[M_{\mathcal I_1}(W,h_n)]+2\tilde C_1\log^{(2)} n}
%    \varepsilon_0^{-1}e^{-h_n+2\tilde C_1\log^{(2)}n}= 2e^{2\varepsilon_0}n^{-1}(\log n)^{-5C_2+2\tilde C_1}
\\
& \leq &
    e^{-h_n+2\tilde C_1\log^{(2)} n}
=
    (\log n)^{5C_2+2\tilde C_1}/n
\leq
    (\log n)^{-2}
    /(3n(\log n)^{-6C_2})\, ,
%    /(2\alpha_n)\, ,
\end{eqnarray*}
for every $\ell\in \mathbb N$ and $\omega\in\mathcal{G}_n^+$ for every $n$ large enough, since $C_2\geq 2\alpha+2+10\,\tilde C_1$.
Summing over $\ell$ proves \eqref{eqpourE4} in this case, the other case $E_2^+(\mathcal I_1)\setminus E_2^-(\mathcal I_1)$ being very similar.

Due to \eqref{eqpourE4},
for large $n$, for every $\omega\in\mathcal{G}_n^+$,
by the strong Markov property,
\begin{eqnarray}
\nonumber
    P_\omega(E_2^c\cap E_4)
& = &
    E_\omega\big[
                %1_{\{\tau(b_{\mathcal I_1})\leq \alpha_n\}}
                1_{E_2^c}
                P_\omega^{b_{\mathcal I_1}}\big(\tau(A)\in[n-2n(\log n)^{-6C_2}-x,n-x]
                \big)_{|x=\tau(b_{\mathcal I_1})}\big]
\\
& \leq &
    (\log n)^{-2}.
\label{eqProbaE4}
\end{eqnarray}
\item  \underline{Control on $E_7$}.
Let us prove that for $n$ large enough,
\begin{equation}\label{eqE7}
    \forall\omega\in\mathcal{G}_n^+,
\qquad
    P_{\omega}(E_2^c\cap E_7)
\leq
    (\log n)^{-2}.
\end{equation}
Due to the strong Markov property, it is enough to prove that for large $n$, for every $\omega\in\mathcal{G}_n^+$, uniformly on $E_2^c$,
\begin{equation}\label{eqLemmaTempsAtteintebI2}
    P_\omega^A\left[\tau(b_{\mathcal I_2})\ge n(\log n)^{-6C_2}\right]
\leq
    (\log n)^{-2}.
\end{equation}
Recall that $h_n$-extrema are a fortiori $(h_n-2C_2\log^{(2)} n)$-extrema.
Let us observe that due to \eqref{eqDefGn1} and since $W(B)-W(b_{\mathcal I_1}(W,h_n))\leq \log n+C_2\log^{(2)} n$,
the only possible slope $T_j[W, h_n-2C_2\log^{(2)} n]$, $-5\leq j\leq 5$ with height $\leq \log n+C_2\log^{(2)} n$
is $[B, b_{\mathcal I_1}(W,h_n)]$ (or $[b_{\mathcal I_1}(W,h_n), B]$)
so $W(B)-W(b_{\mathcal I_2}(W,h_n))>\log n+C_2\log^{(2)} n$.
For the same reason,  there is no
$(h_n-2C_2\log^{(2)} n)$-extrema between $B$ and $b_{\mathcal I_2}(W,h_n)$,
and so $\sup_{B\leq u \leq v\leq b_{\mathcal I_2}(W,h_n)}(W(v)-W(u))<h_n-2C_2\log^{(2)} n$ in the case $\mathcal I_2=\mathcal I_1+1$.
Hence in this case, due to \eqref{InegEsperance1}, \eqref{InegHypothesesMiPetits}, \eqref{KMT1} and to Markov's inequality,
and since $\lfloor B\rfloor=M_{\mathcal I_1}<A<b_{\mathcal I_2}$,
\begin{eqnarray}
P_\omega^{A}\left[\tau(\lfloor B\rfloor) \wedge \tau(b_{\mathcal I_2})
                    \geq n(\log n)^{-6C_2}
            \right]
%    \tau(B)>\alpha_n\right]
&\leq&
    \frac{\varepsilon_0^{-1}4(\log n)^{2\alpha}
    e^{h_n-2C_2\log^{(2)}n+2\tilde C_1\log^{(2)} n}}{n(\log n)^{-6C_2}}
%    e^{h_n-C_2\log_2n+2\tilde C_1\log_2 n}}{2n(\log n)^{-C_2}}
\nonumber\\
&\leq&
    \varepsilon_0^{-1}4(\log n)^{2\alpha-C_2+2\tilde C_1}\le
    \frac 12(\log n)^{-2}\, ,
\label{InegTempsAtteinteBbI2}
\end{eqnarray}
for every $n$ large enough since $C_2\geq 2\alpha+2+10\,\tilde C_1$.
%since $2\alpha-5C_2+2\tilde C_1<-1$.
This is also true in the case $\mathcal I_2=\mathcal I_1-1$ by \eqref{InegEsperance2}.
Moreover in the case $\mathcal I_2=\mathcal I_1+1$,
we have
$
    \max_{[A,b_{\mathcal I_2}]}V
\leq
    \sup_{[M_{\mathcal I_1}^+, b_{\mathcal I_2}(W,h_n)]}W+\tilde C_1\log^{(2)} n
\leq
    W(M_{\mathcal I_1}^+)+(h_n-2C_2\log^{(2)} n)+\tilde C_1\log^{(2)} n
$
due to the previous remark, \eqref{InegHypothesesMiPetits} and \eqref{KMT1}. Also,
$
    W(M_{\mathcal I_1}^+)
\leq
    W(b_{\mathcal I_1}(W,h_n))+C_2\log^{(2)} n
$ by \eqref{eqDefMj+},
otherwise we would have $M_{\mathcal I_1}^+=b_{\mathcal I_1+1}$ and
$
    W(b_{\mathcal I_1+1})
\geq
    W(b_{\mathcal I_1}(W,h_n)) +C_2\log^{(2)} n
\geq
    W(M_{\mathcal I_1}(W,h_n))-\log n
$
due to our hypothesis in this case $\mathcal I_2=\mathcal I_1+1$,
which in turn would give
$
    W(M_{\mathcal I_1}(W,h_n))
-
    W(b_{\mathcal I_1+1}(W,h_n))
\leq
    \log n+2\tilde C_1\log^{(2)}n
$,
which contradicts \eqref{eqDefGn1} since $2\tilde C_1<C_2$.
So by \eqref{probaatteinte}, \eqref{InegHypothesesMiPetits} and \eqref{KMT1},
recalling that $W(b_{\mathcal I_1}(W,h_n))+h_n\leq W(M_{\mathcal I_1}(W,h_n))$ and $B=M_{\mathcal I_1}$, we get
\begin{eqnarray}
    P_\omega^{A}\left[\tau(\lfloor B\rfloor )<\tau(b_{\mathcal I_2})\right]
& \leq &
    \big(b_{\mathcal I_2}-A\big)
    \exp\Big[\max_{[A,b_{\mathcal I_2}]}V-V(B)\Big]
\nonumber\\
& \leq &
    2(\log n)^\alpha e^{\big(W(b_{\mathcal I_1}(W,h_n))+h_n+(\tilde C_1-C_2)\log^{(2)} n\big)
   -\big(W(M_{\mathcal I_1}(W,h_n))-\tilde C_1\log^{(2)} n\big)}
\nonumber\\
& \leq &
    2(\log n)^{\alpha+2\tilde C_1-C_2}
\le
    (\log n)^{-2}/2
\label{InegProbaAtteinteBbI2}
\end{eqnarray}
for every $n$ large enough
since $C_2\geq 2\alpha+2+10\,\tilde C_1$.
We prove similarly \eqref{InegProbaAtteinteBbI2} in the case $\mathcal I_2=\mathcal I_1-1$.
Then, \eqref{InegTempsAtteinteBbI2} and \eqref{InegProbaAtteinteBbI2}
prove \eqref{eqLemmaTempsAtteintebI2}.
Finally, \eqref{eqLemmaTempsAtteintebI2} combined with the strong Markov property lead to \eqref{eqE7}.

\item \underline{Control on $E_5$}.
On $E_1\cap E_2^c\cap E_5$, we have
$
    \tau(b_{\mathcal I_1})
\leq
    n
\leq
    \tau(b_{\mathcal I_1})+\tau(b_{\mathcal I_1}, A)\wedge\tau(b_{\mathcal I_1}, D)
$ and in particular
$Z_n\in[A\wedge D,A\vee D]= [M_{\mathcal I_1}^-, M_{\mathcal I_1}^+]$.
Applying \eqref{eq41Amelioree} as in the simplest case,
%(on $E_1\cap E_2$),
we get for large $n$,
for all $\omega\in\mathcal{G}_n$,
\begin{eqnarray}
&&
    P_\omega(E_1\cap E_2^c\cap E_5, Z_n\notin \Xi_n(W)   )
\nonumber\\
& \leq &
    E_\omega\big[1_{\{\tau(b_{\mathcal I_1})\leq \alpha_n\}}
    P_\omega^{b_{\mathcal I_1}}\big(Z_{n-k}\in [M_{\mathcal I_1}^-, M_{\mathcal I_1}^+]
                            \setminus \Xi_{n}(W)\big)_{|k=\tau(b_{\mathcal I_1})}\big]
\leq
    (\log n)^{-2}.
\label{eqConclusionE5}
\end{eqnarray}
\item \underline{Control on $E_3$}.
On $E_1\cap E_2^c\cap E_3$, we have
$
    \tau(b_{\mathcal I_2})
\leq
    \tau(b_{\mathcal I_1})+\tau(b_{\mathcal I_1}, A)+\tau'(A, b_{\mathcal I_2})
%\leq
%    n-n(\log n)^{-6C_2}
<n
$.
Moreover, the height of the valley $[M_{\mathcal I_2-1}(W,h_n), M_{\mathcal I_2}(W,h_n)]$
is at least $\log n+C_2\log^{(2)} n$ on $E_2^c$ due to \eqref{eqDefGn1}, that is, we are on $E_2^-(\mathcal I_2)\cap  E_2^+(\mathcal I_2)$.
Also we get
$
    P_\omega^{b_{\mathcal I_2}}
    \big[\tau\big(M_{\mathcal I_2-1}\big)\wedge\tau\big(M_{\mathcal I_2}\big)<n\big]
\leq
    2(\log n)^{-2}
$
by \eqref{PPPP1} and \eqref{PPPP2} uniformly on $E_2^c\cap \mathcal{G}_n^+$ for large $n$.
Using \eqref{eq41Amelioree} and $[M_{\mathcal I_2-1}, M_{\mathcal I_2}]\subset[M_{\mathcal I_2}^-, M_{\mathcal I_2}^+]$,
this gives for large $n$ for every $\omega\in\mathcal{G}_n^+$,
\begin{eqnarray}
&&
    P_\omega(E_1\cap E_2^c\cap E_3, Z_n\notin \Xi_n(W)   )
\nonumber\\
& \leq &
    E_\omega\big[1_{\{\tau(b_{\mathcal I_2})< n\}}
    P_\omega^{b_{\mathcal I_2}}\big(Z_{n-k}\in [M_{\mathcal I_2}^-, M_{\mathcal I_2}^+]
                            \setminus \Xi_{n}(W)\big)_{|k=\tau(b_{\mathcal I_2})}\big]
\nonumber\\
&&
    +
    E_\omega\big[1_{E_2^c}
    P_\omega^{b_{\mathcal I_2}}
    \big[\tau\big(M_{\mathcal I_2-1}\big)\wedge\tau\big(M_{\mathcal I_2}\big)<n\big]
    \big]
\nonumber\\
& \leq &
    3(\log n)^{-2}.
\label{eqConclusionE3}
\end{eqnarray}
\end{itemize}
Finally, \eqref{eqE2Union} and the controls on $E_i$, $3\leq i \leq 7$ prove Lemma \ref{lemE2compl},
which ends the proof of Proposition \ref{LemmaLocalisationVerticale}.
\end{proof}

%%%%%%%%%%%%%%%%%%%%%%%%%%%%%%%%%%%%%%%%%%%%%%%%%%%%%%%%%%%%%%%%%%%%%%%%%%%%%%%%%%%%%%%%
%                                                                                      %
%                                   PREUVE Lemme A1                                    %
%                                                                                      %
%%%%%%%%%%%%%%%%%%%%%%%%%%%%%%%%%%%%%%%%%%%%%%%%%%%%%%%%%%%%%%%%%%%%%%%%%%%%%%%%%%%%%%%%

\section{Probability of simultaneous meeting of independent recurrent rwre in the same environment}\label{sameenv}

This section is devoted to the proof of
Proposition \ref{Lemmeannexe1}, which is a consequence of the following proposition whose proof is deferred.

Let $r>1$ and let $Z^{(1)},...,Z^{(r)}$ be $r$ independent recurrent RWRE in the same environment $\omega$ satisfying \eqref{eqHypothesesSinai}.

\begin{prop}\label{Lemmeannexe1Bis}
Let $\delta\in(0,1)$.
There exist events $\Delta_N(\delta)$, $N\geq 1$ and $\widehat b(N)\in 2\Z$ depending only on the environment $\omega$,
and constants $c(\delta)>0$, $\e(\delta)\in(0,1)$, with
\begin{equation}\label{eqProbaLiminfBN}
    \liminf_{N\to+\infty}\p\big[\Delta_N(\delta)\big]
\geq
    1-\delta,
\end{equation}
such that
$$
    \forall (y_1, \dots, y_r)\in (2\Z)^r,
    \exists N_1\in\N,
    \forall N\geq N_1,
    \forall \omega\in \Delta_N(\delta),
    \forall j\in\{1,\dots, r\},
$$
\begin{equation}\label{PZ=b}
   \forall n\in\big[N^{1-\e(\delta)}, N\big]\cap(2\N),
\quad
    P_\omega^{y_j}\big[Z_n^{(j)}=\widehat b(N)\big]\geq c(\delta).
\end{equation}
This remains true if $(2\Z)^r$ and $2\N$ are replaced respectively by
$(2\Z+1)^r$ and $2\N+1$.
\end{prop}
\begin{proof}[Proof of Proposition \ref{Lemmeannexe1}]
%Assume that we have proved Proposition \ref{Lemmeannexe1Bis}.
Let $\delta\in(0,1)$. First, notice that by \eqref{eqProbaLiminfBN},
%where i.o. means infinitely often, i.e. for infinitely many $N$,
%$$
%    \p\big[(B_N(\delta) \text{ i.o.})^c\big]
%=
%    \p\bigg(\bigcup_{N\in\N}\bigcap_{n\geq N }[B_n(\delta)]^c\bigg)
%=
%    \lim_{N\to+\infty}\p\bigg[\bigcap_{n\geq N }[B_n(\delta)]^c\bigg]
%\leq
%    \limsup_{N\to+\infty} \p[B_N(\delta)^c]
%\leq
%$$
\begin{equation}\label{eqProbaLiminfBNio}
%    \p\big[(\Delta_N(\delta) \text{ i.o.})\big]
%=
    \p\big[\limsup_{N\to+\infty}\Delta_N(\delta)\big]
=
    \p\left[\bigcap_{N\in\N}\bigcup_{n\geq N }\Delta_n(\delta)\right]
=
    \lim_{N\to+\infty}\p\bigg[\bigcup_{n\geq N }\Delta_n(\delta)\bigg]
\geq
    \liminf_{N\to+\infty} \p[\Delta_N(\delta)]
\geq
    1-\delta.
\end{equation}
Now, let $(y_1, \dots, y_r)\in (2\Z)^r$.
There exists $N_1\in\N$ such that for every $N\geq N_1$, on $\Delta_N(\delta)$,
$$
    \sum_{n=1}^N\frac 1n
    \sum_{k\in\mathbb Z}\prod_{j=1}^r P_\omega^{y_j}\left[Z_n^{(j)}=k\right]
\geq
    \sum_{n=N^{1-\e(\delta)}}^N\frac {{\bf 1}_{2\N}(n)}n
    [c(\delta)]^r
\geq
    [c(\delta)]^r \frac{\e(\delta)}{4}\log N
$$
if $N$ is large enough.
Consequently, we have on $\limsup_{N\to+\infty} \Delta_N(\delta)=\{\omega\in\Delta_N(\delta) \text{ i.o.}\}$,
%on $\{\Delta_N(\delta) \text{ i.o.}\}$,
$$
    \limsup_{N\to+\infty}\frac{1}{\log N}
    \sum_{n=1}^N\frac 1n
    \sum_{k\in\mathbb Z}\prod_{j=1}^r P_\omega^{y_j}\left[Z_n^{(j)}=k\right]
\geq
    [c(\delta)]^r \frac{\e(\delta)}{4}
>
    0.
$$
This and \eqref{eqProbaLiminfBNio} prove Proposition \ref{Lemmeannexe1} in the case
$(y_1, \dots, y_r)\in (2\Z)^r$. The proof in the case $(y_1, \dots, y_r)\in (2\Z+1)^r$
is similar.
\end{proof}
Now, it remains to  prove Proposition \ref{Lemmeannexe1Bis}.

\subsection{Main idea of the proof of Proposition \ref{Lemmeannexe1Bis}}
Let $Z$ be a RWRE as in Section \ref{SinaiEstimate}.
In order to prove that $Z_n$ is localized at $\widehat b(N)$ with a quenched probability $P_\omega^{y_j}$ greater than a positive constant, we use a coupling argument between a copy of
$Z$ starting from $\widehat b(N)$ and a RWRE $\widehat Z$ reflected in some valley around $\widehat b(N)$, under its invariant probability measure.
To this aim, we approximate the potential $V$ by a Brownian motion $W$, use $W$ to build the set of good environments $\Delta_N(\delta)$ and estimate its probability
$\p\big[\Delta_N(\delta)\big]$, and then define $\widehat b(N)$.

We build $\Delta_N$ as the intersection of $7$ events $\Delta_N^{(i)}$,
$i=0,\dots, 6$.
First, $\Delta_N^{(0)}$ gives an approximation of $V$ by $W$.
Loosely speaking $\Delta_N^{(1)}$ guarantees that the central valley (containing the origin) of height $\log n$
has a height much larger than $\log n$, so that $Z$ will not escape this valley before time $n$ (see Lemma \ref{LemmaProaResterDansVallee}).
$\Delta_N^{(1)}$ also ensures that
this central valley
%the valleys of height at least $\log n$ do
does not contain  sub-valleys of height close to $\log n$, so that with high quenched probability, $Z$ reaches quickly the bottom  of this valley without being trapped in such subvalleys (see Lemma \ref{LemmaTempsAtteintebN}).
To this aim, we also need that the bottom of this valley is not too far from $0$,
which is given by $\Delta_N^{(3)}$, and that the value of the potential between $0$
and the bottom of this valley is
%significatively lower than the closest $(\log N)$-maximum
low enough, which is given by $\Delta_N^{(4)}$ and $\Delta_N^{(2)}$.
Additionally, $\Delta_N^{(5)}$ is useful to provide estimates for the invariant probability measure $\widehat \nu$,
and is useful to prove that the coupling occurs quickly
(Lemma \ref{Lemma_Time_Coupling}, using Lemmas \ref{LemInegaliteSupAtteinteLpm} and
\ref{LemMajorationNuBordsVallee}).
Finally, $\Delta_N^{(6)}$ says that $\widehat \nu\big(\widehat b(N)\big)$,
which is roughly the invariant probability measure at the bottom of the central valley,
is larger than a positive constant.

\subsection{Construction of $\Delta_N(\delta)$}
Let $\delta\in(0,1)$.
The aims of this section are the construction of the set of environments $ \Delta_N(\delta)$ satisfying \eqref{eqProbaLiminfBN} and
\eqref{PZ=b},
and the proof of \eqref{eqProbaLiminfBN}. We will construct $ \Delta_N(\delta)$ as an intersection
\begin{equation}\label{Deltandef}
    \Delta_N(\delta):=\bigcap_{i=0}^6 \Delta_N^{(i)},
\end{equation}
where the sets $\Delta_N^{(i)}$, defined below, also depend on $\delta$.
In what follows, $\e_i$ is for $i>0$ a positive constant depending on $\delta$
and used to define the set $\Delta_N^{(i)}$.
As in the previous section, we will approximate the potential $V$ by a two-sided Brownian motion $W$
such that $\var (W(1))=\var (V(1))$  (see Figure \ref{figure1}
for patterns of the potential $V$ and of $W$ in $\Delta_N(\delta)$).
We start with $\Delta_N^{(1)},\dots,\Delta_N^{(5)}$ which are $W$-measurable.
Using the same notation as before for $h$-extrema, for a two-sided Brownian motion $W$,
we define
\begin{equation}\label{DeltaN1def}
    \Delta_N^{(1)}
:=
    \{W\in\mathcal{W}\}\cap
    \bigcap_{i=-1}^1\big\{H[T_i(W, (1-2\e_1)\log N)]\geq (1+2\e_1)\log N\big\},
\end{equation}
\begin{equation}\label{DeltaNRdef}
    \Delta_N^{(R)}
 :=
    \big\{
        x_1\big( W, (1-2\e_1)\log N\big) \text{ is a } ((1-2\e_1)\log N)\text{-minimum for } W
    \big\},
\end{equation}

and $\Delta_N^{(L)} := \big[\Delta_N^{(R)}\big]^c$,
where $R$ stands for right and $L$ for left,
$\Delta_N^{(2)}:=\Delta_N^{(2,R)}\cup \Delta_N^{(2,L)}$ with
\begin{equation}\label{DeltaN2Rdef}
    \Delta_N^{(2,R)}
:=
    \bigg\{
            \max_{\big[0,  x_1\big(W, (1-2\e_1)\log N\big)\big]} W
            <
            W\big(x_0\big(W, (1-2\e_1)\log N\big)\big)
            -\e_2\log N
            %,
            %x_0\big(\widetilde W_N, 1-2\e\big) \text{ is } (1-2\e) \text{ max}
    \bigg\}
    \cap \Delta_N^{(R)},
\end{equation}
\begin{equation}\label{DeltanNdef}
    \Delta_N^{(2,L)}
:=
    \bigg\{
            \max_{\big[x_0\big(W, (1-2\e_1)\log N\big),0\big]} W
            <
            W\big(x_1\big(W, (1-2\e_1)\log N\big)\big)
            -\e_2 \log N
            %,
            %x_1\big(\widetilde W_N, 1-2\e\big) \text{ is } (1-2\e) \text{ max}
    \bigg\}
    \cap \Delta_N^{(L)};
\end{equation}

\begin{equation}\label{DeltaN3def}
\Delta_N^{(3)}
:=
    \Big\{
          -\e_3^{-1}(\log N)^2
        \leq
          x_{-1}[W,(1-2\e_1)\log N]
        \leq
          x_{2}[W,(1-2\e_1)\log N]
        \leq
          \e_3^{-1}(\log N)^2
    \Big\};
\end{equation}
\begin{equation}\label{DeltaN4def}
    \Delta_N^{(4)}
:=
    \cap_{i=0}^1\{|W(x_i(W,(1-2\e_1)\log N))|>\e_4\log N\}
\end{equation}
and
$\Delta_N^{(5)}:=\Delta_N^{(5,R)}\cup \Delta_N^{(5,L)}$, where
\begin{equation}\label{DeltaN5def}
\Delta_N^{(5,L)}
:=
    \bigg\{
            \min_{\big[0,  x_1\big(W, (1-2\e_1)\log N\big)\big]} W
            >
            W\big(x_0\big(W, (1-2\e_1)\log N\big)\big)
            +\e_5\log N
            %,
            %x_0\big(\widetilde W_N, 1-2\e\big) \text{ is } (1-2\e) \text{ max}
    \bigg\}
    \cap \Delta_N^{(L)},
\end{equation}

\begin{equation}\label{DeltaN5Rdef}
    \Delta_N^{(5,R)}
:=
    \bigg\{
            \min_{\big[x_0\big(W, (1-2\e_1)\log N\big),0\big]} W
            >
            W\big(x_1\big(W, (1-2\e_1)\log N\big)\big)
            +\e_5 \log N
            %,
            %x_1\big(\widetilde W_N, 1-2\e\big) \text{ is } (1-2\e) \text{ max}
    \bigg\}
    \cap \Delta_N^{(R)}.
\end{equation}

\smallskip

\begin{lem}\label{LEMME000}
Let $W$ be a two-sided Brownian motion such that $\var (W(1))=\var (V(1))$.
There exist $(\varepsilon_1,\varepsilon_2,\varepsilon_3,\varepsilon_4, \varepsilon_5)\in(0,1/10)^4$
with $\varepsilon_5=\varepsilon_2$ such that, for every $i\in\{1,...,5\}$,
$\p [\Delta_N^{(i)}]>1-\delta/10 $.
\end{lem}

\begin{figure}[htbp]
\includegraphics[scale=0.837]{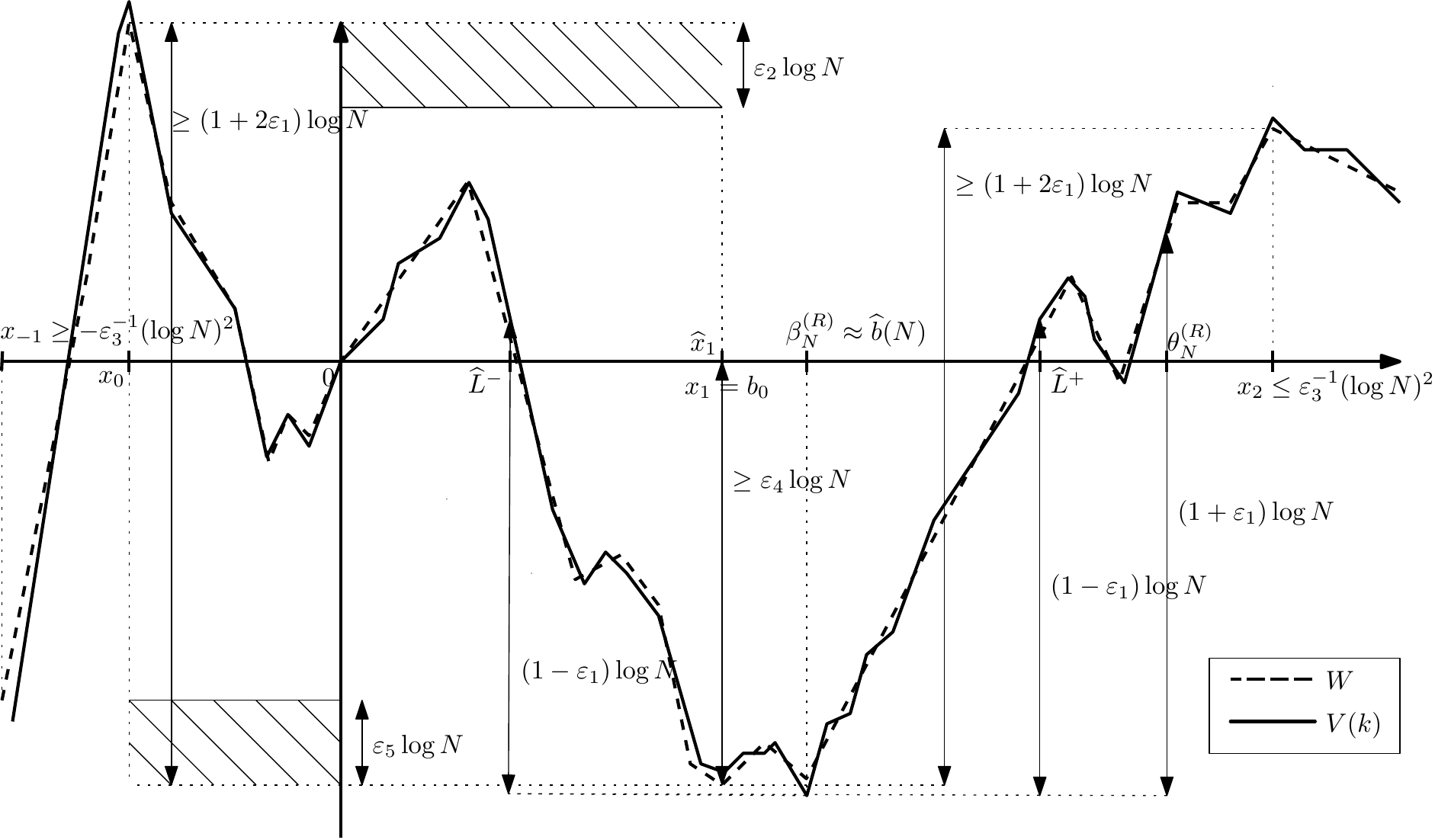}
\caption{Pattern of the potential $V$ and of $ W$ for $\omega\in\Delta_N\cap \Delta_N^{(R)}$,
where $x_i$ denotes $x_i(W, (1-2\varepsilon_1)\log N)$.}
\label{figure1}
\end{figure}

\begin{proof}
First, by the same arguments as in the proof of Lemma \ref{LemmaCardinalHTiPetits}, there exists
$\e_1\in(0,1/10)$ such that
$\p\big[\Delta_N^{(1)}\big]\geq 1-\delta/10$.

We now introduce $\widetilde W_N(x):=W(x(\log N)^2)/\log N$, which has the same law as $W$ by scaling.
We notice that $x_0\big(\widetilde W_N, 1-2\e_1\big)$ is a local extremum for $\widetilde W_N$, so
$\p\big[x_0\big(\widetilde W_N, 1-2\e_1\big)=0\big]=0$. Hence we have
$x_0\big(\widetilde W_N, 1-2\e_1\big)<0<x_1\big(\widetilde W_N, 1-2\e_1\big)$ a.s.
We start with the case where $x_1\big(\widetilde W_N, 1-2\e_1\big)$ is a
$(1-2\e_1)$-minimum for $\widetilde W_N$,
that is, the bottom $b_0(W,(1-2\e_1)\log N)$ of the central valley of depth at least $(1-2\e_1)\log N$ for $W$
is on the right.
That is, we assume we are on $\Delta_N^{(R)}\cap \mathcal{W}$.
Since $\widetilde W_N$ is continuous on $\big[0,  x_1\big(\widetilde W_N, 1-2\e_1\big)\big]$,
$\widetilde W_N$ attains its maximum on this interval at some $y\in\big[0,  x_1\big(\widetilde W_N, 1-2\e_1\big)\big]$.
So, $\widetilde W_N(y)\in\big[0, \widetilde W_N\big(x_0\big(\widetilde W_N, 1-2\e_1\big)\big)\big]$,
since
$
%\max_{T_0(\widetilde W_N, 1-2\e_1)}\widetilde W_N
\max_{[x_0(\widetilde W_N, 1-2\e_1),x_1(\widetilde W_N, 1-2\e_1)]}\widetilde W_N
=\widetilde W_N\big(x_0\big(\widetilde W_N, 1-2\e_1\big)\big)$.
If $\widetilde W_N(y)= \widetilde W_N\big(x_0\big(\widetilde W_N, 1-2\e_1\big)\big)$,
then $y$ would be
a $(1-2\e_1)$-maximum for $\widetilde W_N$, with
$x_0\big(\widetilde W_N, 1-2\e_1\big)<y<x_1\big(\widetilde W_N, 1-2\e_1\big)$, which is not possible on $\mathcal{W}$.
So,
$
    \widetilde W_N(y)
=
    \max_{\big[0,  x_1\big(\widetilde W_N, 1-2\e_1\big)\big]} \widetilde W_N
<
    \widetilde W_N\big(x_0\big(\widetilde W_N, 1-2\e_1\big)\big)
$.
Consequently, there exists $\e_2\in(0,1/10)$ such that
$$
    \p\left[
            \max_{\big[0,  x_1\big(\widetilde W_N, 1-2\e_1\big)\big]} \widetilde W_N
            <
            \widetilde W_N\big(x_0\big(\widetilde W_N, 1-2\e_1\big)\big)
            -\e_2
      \bigg|
      %x_0\big(\widetilde W_N, 1-2\e\big) \text{ is } (1-2\e) \text{ max}
      \Delta_N^{(R)}
      \right]
\geq
    1-\frac{\delta}{10},
$$
%where max means maximum,
and the same is true if we exchange $x_0$ and $x_1$ by symmetry
(and then $[0, x_1(\dots)]$ is replaced by $[x_0(\dots), 0]$, and $\Delta_N^{(R)}$ by $\Delta_N^{(L)}$).
Hence
$\p\big[\Delta_N^{(2)}\big]\geq 1-\delta/10$ by scaling.

Moreover, there exists $\e_3\in(0,1/10)$ such that
$
    \p\big[\Delta_N^{(3)}\big]
=
    \p\big[-\e_3^{-1}\leq x_{-1}[\widetilde W_N,1-2\e_1]\leq x_{2}[ \widetilde W_N,1-2\e_1] \leq \e_3^{-1}
\big]
\geq
    1-\delta/10
$, where we get the first equality by scaling.

Finally, there exists $\e_4\in(0,1/10)$ such that $\p\big[\Delta_N^{(4)}\big]\geq 1-\delta/10$,
by scaling, since $\big|\widetilde W_N\big(x_i\big(\widetilde W_N, 1-2\e_1\big)\big)\big|>0$ a.s. for $i\in\{0,1\}$.
Indeed, $x_0\big(\widetilde W_N, 1-2\e_1\big)<0$,
so $\widetilde W_N\big(x_0\big(\widetilde W_N, 1-2\e_1\big)\big)=\max_{[x_0(\widetilde W_N, 1-2\e_1),0]}\widetilde W_N>0$ a.s.
on $\Delta_N^{(R)}\cap \mathcal{W}$,
and
$\widetilde W_N\big(x_0\big(\widetilde W_N, 1-2\e_1\big)\big)=\min_{[x_0(\widetilde W_N, 1-2\e_1),0]}\widetilde W_N<0$ a.s.
on $\Delta_N^{(L)}\cap \mathcal{W}$,
so $\big|\widetilde W_N\big(x_0\big(\widetilde W_N, 1-2\e_1\big)\big)\big|>0$ a.s.
Similarly, $\big|\widetilde W_N\big(x_1\big(\widetilde W_N, 1-2\e_1\big)\big)\big|>0$ a.s.

Replacing $W$ by $-W$ in $\Delta_N^{(2)}$ proves that with $\e_5:=\e_2>0$, the event $\Delta_N^{(5)}$
satisfies
$\p\big[\Delta_N^{(5)}\big]=\p\big[\Delta_N^{(2)}\big]\geq 1-\delta/10$.
\end{proof}
From now on, $\varepsilon_1,...,\varepsilon_5$ are the ones given by Lemma \ref{LEMME000}. Let
\begin{equation}\label{epsilondef}
\varepsilon:=\min(\varepsilon_1,\dots,\varepsilon_5)/9.
\end{equation}
\begin{lem}\label{lemmeKMT2}
Up to an enlargement of $(\Omega,\mathcal F,\p)$, there exist a two-sided
Brownian motion $(W(s),\ s\in\R)$ defined on $\Omega$ such that
$\var(W(1))=\var(V(1))$ and a real number $\xi>0$ such that
$$
    \p\left[\sup_{|t|\le 2\varepsilon_3^{-1}(\log N)^2}\big|V(\lfloor t\rfloor)-W(t)\big| >\varepsilon\log N\right]
=
    O(N^{-\xi}).
$$
\end{lem}
\begin{proof}
Due to \eqref{KMT} (applied with $N$ replaced by $2\e_3^{-1}(\log N)^2$
and $x=(\e/2)\log N-a\log[2\e_3^{-1}(\log N)^2]$), there exists  for $N$ large enough,
possibly on an enlarged probability space,
a Brownian motion $(W(s), \,s\in\R)$ such that
$$
    \p\left[\sup_{|k|\le 2\varepsilon_3^{-1}(\log N)^2}|V(k)-W(k)| >\frac\varepsilon 2\log N\right]\le N^{-c\frac{\varepsilon}{10}}
$$
and such that $\var(W(1))=\var(V(1))$.
Moreover,
\begin{eqnarray*}
    \p\left[\sup_{|t|\le 2\varepsilon_3^{-1}(\log N)^2}\big|W(t)-W(\lfloor t\rfloor)\big| >\frac\varepsilon 2\log N\right]
&\le&
    5\varepsilon_3^{-1}(\log N)^2\p\left[\sup_{|t|\le 1}|W(t)| >\frac\varepsilon 2\log N\right]
\\
&=&
    O((\log N)^2 \exp[-\varepsilon^2(\log N)^2/(8\sigma^2)]).
\end{eqnarray*}
Combining these two inequalities proves the lemma.
\end{proof}

Recall \eqref{Deltandef}: it remains to define $\Delta_N^{(0)}$ and $\Delta_N^{(6)}$,
see \eqref{eq_def_Delta0} and \eqref{eq_def_Delta6} below, and then we claim
\begin{lem}\label{LemmaProbaDeltaN}
For large $N$, $\p\big[\Delta_N(\delta)\big]\geq 1-\delta$. Hence \eqref{eqProbaLiminfBN} holds true.
\end{lem}
\begin{proof}
From now on, $W$ is the Brownian motion $W$ coming from Lemma \ref{lemmeKMT2} and  $\Delta_N^{(1)},...,\Delta_N^{(5)}$ are the corresponding events defined in %\eqref{DeltaN2Rdef}--\eqref{DeltaN5Rdef}.
\eqref{DeltaN1def}--\eqref{DeltaN5Rdef}.
We set
\begin{equation}\label{eq_def_Delta0}
    \Delta_N^{(0)}:= \left\{\sup_{|t|\le 2\varepsilon_3^{-1}(\log N)^2}\big|V(\lfloor t\rfloor)-W(t)\big| \le \varepsilon\log N\right\}.
\end{equation}
For $N$ large enough, $\p\big[\Delta_N^{(0)}\big]>1-\delta/10$ by Lemma \ref{lemmeKMT2}.
In particular on the event $\Delta_N^{(0)}\cap \Delta_N^{(3)}$, we can apply the inequalities of $\Delta_N^{(0)}$
to any $t\in\big[x_{-1}\big(W, (1-2\varepsilon_1)\log N\big), x_{2}\big(W, (1-2\varepsilon_1)\log N\big)\big]$,
since those $t$ satisfy $|t|\leq \e_3^{-1}(\log N)^2$.
We now introduce (here this is for $V$ directly, not for $W$)
\begin{eqnarray*}
    \theta_N^{(R)}
& := &
    \inf\Big\{i\in \N,\ V(i)-\min_{0\leq j\leq i} V(j) \geq (1+\e_1)\log N\Big\},
\\
    \beta_N^{(R)}
& := &
    \sup\Big\{i<\theta_N^{(R)},\ V(i)=\min_{0\leq j \leq \theta_N^{(R)}} V(j)\Big\},
\\
    \theta_N^{(L)}
& := &
    \sup\Big\{i\in (-\N),\ V(i)-\min_{i\leq j\leq 0} V(j) \geq (1+\e_1)\log N\Big\},
\\
    \beta_N^{(L)}
& := &
    \inf\Big\{i>\theta_N^{(L)},\ V(i)=\min_{\theta_N^{(L)}\leq j \leq 0} V(j)\Big\}.
\end{eqnarray*}
By (\cite{Dembo_Gantert_Peres_Shi}, eq. (4.33)), there exists $\e_6>0$ such that if $N$ is large enough,
$\p\big[\Delta_N^{(6,R)}\big]\geq 1-\delta/10$, where
\begin{equation}\label{Delta6}
    \Delta_N^{(6,R)}
:=
    \Bigg\{
        \sum_{i=0}^{\theta_N^{(R)}-1} e^{-\big[V(i)-V\big(\beta_N^{(R)}\big)\big]}
        \leq
        \e_6^{-1}
    \Bigg\},
\quad
    \Delta_N^{(6,L)}
:=
    \Bigg\{
        \sum_{i=\theta_N^{(L)}}^{-1} e^{-\big[V(i)-V\big(\beta_N^{(L)}\big)\big]}
        \leq
        \e_6^{-1}
    \Bigg\}.
\end{equation}
Replacing $V(.)$ by $V(-.)$ gives $\p\big[\Delta_N^{(6,L)}\big]\geq 1-\delta/10$.
Consequently, $\p\big[\Delta_N^{(6)}\big]\geq 1-2\delta/10$,
where
\begin{equation}\label{eq_def_Delta6}
    \Delta_N^{(6)}
:=
    \Delta_N^{(6,R)}\cap \Delta_N^{(6,L)}.
\end{equation}
This, combined with Lemma \ref{LEMME000} and $\p\big[\Delta_N^{(0)}\big]>1-\delta/10$, proves the lemma.
\end{proof}

%%%%%%%%%%%%%%%%%%%%%%%%%%%%%%%%%%%%%%%%%%%%%%%%%%%%%%%%%%%%%%%%%%

\subsection{Random walk in an environment $\omega\in\Delta_N(\delta)$}
The aim of this subsection is to prove Proposition \ref{Lemmeannexe1Bis}
with the $\Delta_N(\delta)$ constructed in the previous section, see \eqref{Deltandef}--\eqref{DeltaN5Rdef},
\eqref{eq_def_Delta0} and \eqref{eq_def_Delta6}.
Let $\delta\in(0,1)$. We write $\Delta_N$ for $\Delta_N(\delta)$.
We also fix $(y_1, \dots, y_r)\in (2\Z)^r$.
There exists $N_2\in\N$ such that for $N\geq N_2$,
$\p\big[\Delta_N\big]\geq 1-\delta$ (due to Lemma \ref{LemmaProbaDeltaN}), $a_0\leq \varepsilon\log N$,
and  $\max_{1\leq j \leq r}|y_j|<  \min(\e_2,\e_4) (\log N)/(4a_0)$,
where we set  $a_0:=\log((1-\varepsilon_0)/\varepsilon_0)$.

We introduce, recalling \eqref{DeltaNRdef},
\begin{equation}\label{eqDefWidehatbN}
    \widehat b(N)
:=
    2\big\lfloor\beta_N^{(R)}/2\big\rfloor{\bf 1}_{\Delta_N^{(R)}}
+
    2\big\lfloor\beta_N^{(L)}/2\big\rfloor{\bf 1}_{\Delta_N^{(L)}}.
\end{equation}
We will carry out the proof in the case $\omega \in \Delta_N \cap \Delta_N^{(R)}$. The case
$\omega \in \Delta_N \cap \Delta_N^{(L)}$ is similar by symmetry.
We define $\widehat x_i:=\lfloor x_i(W, (1-2\e_1)\log N)\rfloor$,
and
$$
\qquad
    D_N^{(1)}
:=
    \big\{\tau\big(\widehat b(N)\big)<\tau(\widehat x_0)\big\},
\qquad
    D_N^{(2)}
:=
    \big\{\tau(\widehat x_0)\wedge\tau\big(\widehat b(N)\big)\leq N^{1-\e_1}\big\}.
$$
We sometimes write $x_i$ instead of $x_i(W,(1-2\e_1)\log N)$ in the following.

In the following lemma, we prove that
$Z$ goes quickly to $\widehat b(N)$, which is nearly the bottom of the potential $V$ in the central valley $\big[\widehat x_0,\widehat x_2\big]$,
with large probability under $P_\omega^{y_j}$,
uniformly on $\Delta_N \cap \Delta_N^{(R)}$ and $j$.

\begin{lem}\label{LemmaTempsAtteintebN}
There exists $N_3\in\N$ such that for all $N\geq N_3$,
\begin{equation*}
    \forall \omega\in \Delta_N \cap \Delta_N^{(R)},
    \forall j\in\{1,\dots , r\},
\quad
    P_\omega^{y_j}\big[D_N^{(1)}\big]
\geq
     1-N^{-(\e_1\wedge \e_2)/4},
\quad
    P_\omega^{y_j}\big[D_N^{(2)}\big]
\geq
     1-N^{-\e_1/4}.
\end{equation*}
\end{lem}

%\smallskip
\begin{proof}
Let $N\geq N_2$, $\omega\in \Delta_N \cap \Delta_N^{(R)}$ and $j\in\{1,\dots , r\}$.
First, notice that
$W(x_2)-W(x_1)=H[T_1(W, (1-2\varepsilon_1)\log N)]\geq (1+2\varepsilon_1)\log N$
because $\omega\in\Delta_N^{(1)}$.
This gives, recalling \eqref{epsilondef}
\begin{equation}\label{InegPreuvex2bN}
    V(\widehat x_2)-V(\widehat x_1)
\geq
    W(x_2)-W(x_1)-2\varepsilon \log N
\geq
    (1+\varepsilon_1)\log N
\end{equation}
since $\omega\in\Delta_N^{(3)}\cap \Delta_N^{(0)}$ (see \eqref{eq_def_Delta0} and the remark after it).
Hence $0\leq \widehat b(N)\leq \beta_N^{(R)}\leq \theta_N^{(R)}\leq \widehat x_2\leq \varepsilon_3^{-1}(\log N)^2$.\\
Now, assume that $\theta_N^{(R)}<x_1$.
Since $V\big(\theta_N^{(R)}\big)-V\big(\beta_N^{(R)}\big)\geq (1+\varepsilon_1)\log N$,
the previous inequalities would give, on $\Delta_N^{(0)}\cap\Delta_N^{(3)}$,
$W\big(\theta_N^{(R)}\big)-W\big(\beta_N^{(R)}\big)\geq (1+\varepsilon_1-2\varepsilon)\log N\geq (1-2\varepsilon_1)\log N$.
So, recalling that $W(x_1)=\min_{[0, x_1]}W$, there would exist a $((1-2\varepsilon_1)\log N)$-maximum
for $W$ in $]0, x_1[$, which is not possible.
Hence $x_1\leq \theta_N^{(R)}$.

%So, $V(\beta_N^{(R)})+2A\leq V(\widehat x_1)+2A\leq W(x_1)+2\varepsilon \log N
%< -\varepsilon_4(\log N)/2$ because $\omega\in \Delta_N^{(0)}\cap\Delta_N^{(3)}\cap \Delta_N^{(4)}$ and $N\geq N_2$.

So, $V\big(\beta_N^{(R)}\big)\leq V(\widehat x_1)\leq W(x_1)+\varepsilon \log N
< -8\varepsilon_4(\log N)/9$ because $\omega\in \Delta_N^{(0)}\cap\Delta_N^{(3)}\cap \Delta_N^{(4)}$.
If $y_j>0$, then
$\min_{[0, y_j]} V\geq -|y_j|a_0\geq -\e_4(\log N)/4>
%V(\widehat x_1)\geq
V(\beta_N^{(R)})+2a_0
$, because $N\geq N_2$.
Since similarly, $\max_{[0, y_j]} V\leq \varepsilon_4(\log N)/4$ and $\varepsilon_4<1$,
we get successively $y_j\leq \theta_N^{(R)}$ and
$y_j\leq \beta_N^{(R)}-2\leq \widehat b(N)-1$.
If $y_j<0$, we prove similarly that $\widehat x_0<y_j$ since $V(\widehat x_0)\geq 8\varepsilon_4(\log N)/9$.
Hence in every case, $\widehat x_0<y_j<\widehat b(N)$.

We now prove that
\begin{equation}\label{InegMaxV}
    \max_{[y_j, \widehat b(N)]} V-V(\widehat x_0)
\leq
    -[(\e_1\wedge\e_2)/2]\log N.
\end{equation}
To this aim, notice that
$
    \max_{[0, \widehat x_1]} V-V(\widehat x_0)
\leq
    -\e_2 (\log N)/2
$
since $\omega\in\Delta_N^{(2,R)}\cap\Delta_N^{(0)}\cap\Delta_N^{(3)}$,
and that if $y_j<0$, we have
$
    \max_{[y_j,0]} V-V(\widehat x_0)
\leq
    |y_j|a_0-(8/9)\e_2\log N
\leq
    -\e_2 (\log N)/2
$
since $W(x_0)\geq \e_2\log N$ on $\Delta_N^{(2,R)}$ and so
$V(\widehat x_0)\geq (8/9)\e_2\log N$. This gives \eqref{InegMaxV}
when $\widehat b(N)\leq \widehat x_1$.\\
Assume now $\widehat x_1<\widehat b(N)$. We have seen after \eqref{InegPreuvex2bN} that $0\leq \widehat b(N)\leq \theta_N^{(R)}\leq \widehat x_2$,
moreover, $V\big(\widehat b(N)\big)\leq V\big(\beta_N^{(R)}\big)+a_0$
and we have proved that $V\big(\beta_N^{(R)}\big)\leq V(\widehat x_1)$,
so we obtain
$$
    V\big(\widehat b(N)\big)-\varepsilon\log N-a_0
\le
    V(\hat x_1)-\varepsilon\log N
\leq
    W(x_1)
\leq
    W\big(\widehat b(N)\big)
\leq
    V\big(\widehat b(N)\big)+\e\log N
$$
since $W(x_1)=\min_{[x_0,x_2]}W$ and $\omega\in\Delta_N^{(3)}\cap \Delta_N^{(0)}$,
so that
\begin{equation}\label{minproches}
    \big|W(x_1)-V\big(\widehat b(N)\big)\big|
\leq
    \varepsilon \log N+a_0
\leq
    2\varepsilon \log N
\leq
   2 \min(\varepsilon_1,\varepsilon_2)(\log N)/9.
\end{equation}
Moreover there is no $((1-2\e_1)\log N)$-maximum for $W$ in $(x_0, x_2)$,
therefore,
\begin{equation}\label{InegMaxW}
    \max_{[x_1, \widehat b(N)]} W
<
    W\big(\widehat b(N)\big)+(1-2\e_1)\log N
\leq
    W(x_1)+(1-2\e_1+3\e_1/9)\log N\, ,
\end{equation}
by $\Delta_N^{(0)}$ applied to $\widehat b(N)$ followed by \eqref{minproches}.
Since $V(\widehat x_0)\geq V(\widehat x_1)+(1+\e_1)\log N$ on $\Delta_N^{(1)}\cap\Delta_N^{(0)}\cap\Delta_N^{(3)}$, this gives
$
    \max_{[\widehat x_1, \widehat b(N)]} V-V(\widehat x_0)
\leq
    -\e_1\log N
$ (since $\omega\in\Delta_N^{(0)}\cap \Delta_N^{(3)}$).
Recapitulating all this gives \eqref{InegMaxV} also when $\widehat x_1<\widehat b(N)$.

So by \eqref{probaatteinte} and \eqref{InegMaxV},
we get uniformly on $\Delta_N\cap\Delta_N^{R)}$ and $j$ for large $N$,
$$
    P_\omega^{y_j}\big[\big(D_N^{(1)}\big)^c\big]
\leq
    \big[\widehat b(N)-y_j\big]\exp\Big[\max_{[y_j, \widehat b(N)]} V-V(\widehat x_0)\Big]
\leq
    \frac{2\e_3^{-1}(\log N)^2}{ N^{(\e_1\wedge \e_2)/2}}
\leq
    N^{-(\e_1\wedge \e_2)/4},
$$
where we used $\omega\in \Delta_N^{(3)}$ and $\widehat x_0< y_j<\widehat b(N)<\widehat x_2$.
This proves the first inequality of the lemma.

We now turn to  $D_N^{(2)}$.
Notice that
$\big|\widehat b(N)-\widehat x_0\big|\leq|\widehat x_2-\widehat x_0|\leq 3\e_3^{-1}(\log N)^2$ on $\Delta_N$
since $0\leq \widehat b(N)\leq \widehat x_2$ as proved after \eqref{InegPreuvex2bN}.
Moreover, there is no $((1-2\e_1)\log N)$-maximum for $W$ in $(x_0, x_1)$,
so $\max_{x_0\leq u \leq v \leq x_1}(W(v)-W(u))<(1-2\varepsilon_1)\log N$.
Also if $x_1<\widehat b(N)$, $\min_{[x_0,\widehat b(N)]} W=W(x_1)$ and \eqref{InegMaxW}
lead to
$\max_{x_1\leq u \leq v \leq \widehat b(N)}(W(v)-W(u))\leq(1-2\varepsilon_1+3\varepsilon_1/9)\log N$.
Since $\omega\in\Delta_N^{(0)}\cap \Delta_N^{(3)}$, this gives
\begin{equation}\label{eqPlusGrandeMonteeEntrex0etbN}
    \max_{\widehat x_0\leq \ell \leq k \leq \widehat b(N)}\big(V(k)-V(\ell)\big)
\leq
    (1-13\varepsilon_1/9)\log N.
\end{equation}
Hence, we have by \eqref{InegEsperance1},
\begin{equation*}
    E_\omega^{y_j}\big[\tau(\widehat x_0)\wedge \tau\big(\widehat b(N)\big)\big]
\leq
    \frac{[\widehat b(N)-\widehat x_0]^2}{\varepsilon_0}
    \exp\Big[\max_{\widehat x_0\leq \ell \leq k \leq \widehat b(N)}\big(V(k)-V(\ell)\big)\Big]
\leq
    \frac{9(\log N)^4N^{1-\frac{13\e_1}{9}}}{\e_0\e_3^{2}}.
\end{equation*}
So due to Markov's inequality,
$
    P_\omega^{y_j}\big[\big(D_N^{(2)}\big)^c\big]
\leq
    N^{-\e_1/4}
$,
uniformly in $\omega\in \Delta_N\cap\Delta_N^{(R)}$ and $j$, for large $N$.
\end{proof}

In the following lemma, we prove that with large quenched probability, uniformly on $\Delta_N \cap \Delta_N^{(R)}$,
after first hitting $\widehat b(N)$, the random walk $Z$ stays in the central valley $\big[\widehat x_0,\widehat x_2\big]$
at least up to time $N$.
To this aim, we now define
$$
    D_N^{(3)}
:=
    \big\{\forall k\in\big[\tau\big(\widehat b(N)\big), \tau\big(\widehat b(N)\big)+N-1\big],\ \widehat x_0<Z_k<\widehat x_2\big\}.
$$

\smallskip
\begin{lem}\label{LemmaProaResterDansVallee}
We have for large $N$,
$$
    \forall \omega \in \Delta_N \cap \Delta_N^{(R)},
    \forall j\in\{1,\dots,r\},
\quad
    P_\omega^{y_j}\big[D_N^{(3)}\big]=P_\omega^{\widehat b(N)}\big[\tau\big(\widehat x_0\big)\wedge\tau\big(\widehat x_2\big)\geq N\big]
%    P_\omega^j\big[D_N^{(3)}\big]=P_\omega^{\widehat b(N)}\big[\forall 0\leq k\leq N, \widehat x_0<Z_k <\widehat x_2\big]
%    P_\omega^{\widehat b(N)}\big[D_N^{(3)}\big]
\geq
    1-2e^{2a_0}N^{-\e_1}.
$$
\end{lem}

\begin{proof}
Let $\omega \in \Delta_N \cap \Delta_N^{(R)}$.
We recall that $|V(k)-V(k-1)|\leq a_0$ for every $k\in\Z$.
We have, since $x_1\leq \theta_N^{(R)}$ and so $V(\beta_N^{(R)})\leq V(\widehat x_1)$, and by \eqref{InegPreuvex2bN},
\begin{equation}\label{InegVx2}
    V\big(\widehat b(N)\big)-V(\widehat x_2)
\leq
    V\big(\widehat x_1\big)+a_0 -V(\widehat x_2)
\leq
    a_0 -(1+\e_1)\log N.
\end{equation}
Similarly,
\begin{equation}\label{InegV0}
    V\big(\widehat b(N)\big)-V(\widehat x_0)
\leq
    a_0 -(1+\e_1)\log N.
\end{equation}
Hence \eqref{InegProba1} and \eqref{InegProba2} lead respectively to
\begin{align*}
&    P_\omega^{\widehat b(N)}(\tau(\widehat x_2)<N)
\leq
    N\exp\Big(\min_{[\widehat b(N),\widehat x_2-1]}V-V(\widehat x_2-1)\Big)
\leq
    N e^{2a_0-(1+\e_1)\log N}
\leq
    e^{2a_0} N^{-\e_1},
\\
&
    P_\omega^{\widehat b(N)}(\tau(\widehat x_0)<N)
\leq
    N\exp\Big(\min_{[\widehat x_0,\widehat b(N)-1]}V-V(\widehat x_0)\Big)
\leq
    N e^{2a_0-(1+\e_1)\log N}
\leq
    e^{2a_0}N^{-\e_1}.
\end{align*}
These two inequalities
%and the strong Markov property
yield
$
%    P_\omega^{\widehat b(N)}\big[\big(D_N^{(3)}\big)^c\big]
    P_\omega^{\widehat b(N)}\big[\tau\big(\widehat x_0\big)\wedge\tau\big(\widehat x_2\big)<N\big]
\leq
    2e^{2a_0}N^{-\e_1}
$, uniformly on $\Delta_N \cap \Delta_N^{(R)}$, which proves the lemma.
\end{proof}

Now, similarly as in Brox \cite{Brox} for diffusions in random potentials (see also \cite[p. 45]{AndreolettiDevulder}), we introduce
a coupling between $Z$ (under $P_\omega^{\widehat b(N)}$) and a reflected random walk $\widehat Z$ defined below.
More precisely, we define, for fixed $N$,
$\widehat \omega_{\widehat x_0}:= 1$,
$\widehat \omega_x:=\omega_x$ if $\widehat x_0<x<\widehat x_2$,
and $\widehat \omega_{x_2} :=0$. We consider a random walk $\big(\widehat Z_n\big)_n$
in the environment $\widehat \omega$, starting from $x\in\big[\widehat x_0, \widehat x_2\big]$,
and denote its law by $P_{\widehat \omega}^x$. That is, $\widehat Z$ satisfies \eqref{eqDefQuenchedLaw_2}
with $\widehat \omega$ instead of $\omega$ and $\omega^{(j)}$ and $\widehat Z$ instead of $Z^{(j)}$.
In words, $\widehat Z$ is a random walk in the environment $\omega$, starting from $x\in\big[\widehat x_0,\widehat x_2\big]$,
and reflected at $\widehat x_0$ and $\widehat x_2$.
Also, let
$$
    \widehat \mu(\widehat x_0)
:=
    e^{-V(\widehat x_0)},
\quad
    \widehat \mu(\widehat x_2)
:=
    e^{-V(\widehat x_2-1)},
\quad
    \widehat \mu(x)
:=
    e^{-V(x)}+e^{-V(x-1)},
\quad
    \widehat x_0<x<\widehat x_2,
$$
and $\widehat\mu(x)=0$ if $x\notin [\widehat x_0,\widehat x_2]$.
Notice that $\widehat\mu(.)/\widehat\mu(\Z)$ is an invariant probability measure for $\widehat Z$.
As a consequence,
\begin{equation}\label{eqDefWidehatNu}
    \widehat \nu(x)
:=
    \widehat \mu(x){\bf 1}_{2\Z}(x)/{\widehat \mu(2\Z)},
\qquad
    x\in\Z,
\end{equation}
is an invariant probability measure for $\big(\widehat Z_{2n}\big)_n$ for fixed $\widehat \omega$.
That is, $P_{\widehat \omega}^{\widehat \nu}\big(\widehat Z_{2k}=x\big)=\widehat \nu(x)$ for every $x\in\Z$ and $k\in\N$, where
$P_{\widehat \omega}^{\widehat \nu}(.):= \sum_{x\in\Z} \widehat \nu(x) P_{\widehat \omega}^{x}(.)$.
Notice that $\widehat \nu$ and $\widehat \mu$ depend on  $N$ and $\omega$.

We can now, again for fixed $N$ and $\omega$,  build a coupling $Q_\omega$ of $Z$ and $\widehat Z$,
such that
\begin{equation}\label{eqLawUnderQ}
    Q_\omega\big(\widehat Z\in .\big)
=
    P_{\widehat \omega}^{\widehat \nu}\big(\widehat Z\in .\big),
\qquad
    Q_\omega(Z\in .)
=
    P_\omega^{\widehat b(N)}(Z\in .),
\end{equation}
such that under $Q_\omega$, these two Markov chains move independently until
$$
    \tau_{\widehat Z=Z}
:=
    \inf\big\{k\geq 0   ,\ \widehat Z_k=Z_k\big\},
$$
which is their first meeting time,
%being a.s. finite since $Z$ is recurrent and $\widehat Z_0-Z_0$ is even,
then $\widehat Z_k=Z_k$ for every $\tau_{\widehat Z=Z}\leq k<\tau_{exit}$, where
$\tau_{exit}$ is the next exit time of $Z$ from the central valley $[\widehat x_0, \widehat x_2]$, that is,
$$
    \tau_{exit}
:=
    \inf\big\{k>\tau_{\widehat Z=Z}, \ Z_k\notin [\widehat x_0, \widehat x_2] \big\},
$$
and then $\widehat Z$ and $Z$ move independently again after $\tau_{exit}$.

Now, we would like to prove that under $Q_\omega$, $Z$ and $\widehat Z$ collide quickly, that is,
$\tau_{\widehat Z=Z}$ is very small compared to $N$. To this aim, we introduce
\begin{eqnarray*}
    \widehat L^-
& := &
    \sup\{k\leq \widehat b(N),\ V(k)-V\big(\widehat b(N)\big)\geq (1-\e_1)\log N\},
\\
    \widehat L^+
& := &
    \inf\{k\geq \widehat b(N),\ V(k)-V\big(\widehat b(N)\big)\geq (1-\e_1)\log N\}.
\end{eqnarray*}
Let $u\vee v:=\max(u,v)$.  We have the following:

\begin{lem}\label{LemInegaliteSupAtteinteLpm}
We have  for large $N$, $\tau(.)$ denoting hitting times by $Z$ as before,
$$
    \forall \omega \in \Delta_N \cap \Delta_N^{(R)},
\qquad
    Q_\omega\big[\tau\big(\widehat L^-\big)\vee\tau\big(\widehat L^+\big)> N^{1-\e_1/2}\big]
\leq
    4 N^{-\e_1/4}.
$$
\end{lem}

\begin{proof}
Let $N\geq N_2$ and $\omega \in \Delta_N \cap \Delta_N^{(R)}$.
Notice that $\widehat x_0\leq \widehat L^-< \widehat b(N)<   \widehat L^+\leq \theta_N^{(R)}\leq \widehat x_2$
similarly as after \eqref{InegPreuvex2bN}.
Because $\omega \in \Delta_N^{(5,R)}\cap\Delta_N^{(0)}\cap \Delta_N^{(3)}$ and due to \eqref{minproches}, we have
since $\varepsilon_5=\varepsilon_2$,
\begin{equation}\label{eqInfVentrex0et0}
    \forall k\in[\widehat x_0, -1],
\quad
    V(k)-V\big(\widehat b(N)\big)
\geq
    W(x_1)+(\varepsilon_5-\varepsilon)\log N-V\big(\widehat b(N)\big)
\geq
    (\e_5/2)\log N.
\end{equation}
Moreover, recalling  $a_0=\log((1-\varepsilon_0)/\varepsilon_0)$,
we have
$
%    \min_{[0, \widehat L^+]}V
%\geq
    \min_{[0, \theta_N^{(R)}]}V
=
    V\big(\beta_N^{(R)}\big)
\geq
    V\big(\widehat b(N)\big)-a_0
$, so
$
    \min_{[\widehat x_0, \widehat L^+]} V
\geq
    \min_{[\widehat x_0, \theta_N^{(R)}]} V
\geq
    V\big(\widehat b(N)\big)-a_0
$.
Notice also for further use that, for every $k\in\big[\theta_N^{(R)},\widehat x_2\big]$,
we have
$
    V\big(\theta_N^{(R)}\big)-V(k)
\leq
    W\big(\theta_N^{(R)}\big)-W(k)+2\varepsilon\log N
<
    (1-2\varepsilon_1+2\varepsilon)\log N
$
since $\omega\in\Delta_N^{(0)}\cap\Delta_N^{(3)}$ and because there is no
$((1-2\e_1)\log N)$--maximum for $W$ in $(\widehat x_1,\widehat x_2\big)$
and $\widehat x_1\leq \theta_N^{(R)}\leq k\leq \widehat x_2$, as proved after \eqref{InegPreuvex2bN}.
Since
$
    V\big(\theta_N^{(R)}\big)-V\big(\widehat b(N)\big)
\geq
    (1+\varepsilon_1)\log N-a_0
$, this gives
\begin{eqnarray}
    \forall k\in\big[\theta_N^{(R)},\widehat x_2\big],
\quad
    V(k)-V\big(\widehat b(N)\big)
& = &
    V(k)-V\big(\theta_N^{(R)}\big)+V\big(\theta_N^{(R)}\big)-V\big(\widehat b(N)\big)
\nonumber\\
& \geq &
    2\e_1\log N.
\label{InegMinVThetaX2}
\end{eqnarray}
Putting together these inequalities gives in particular
$
    \min_{[\widehat x_0, \widehat x_2]} V
\geq
    V\big(\widehat b(N)\big)-a_0
$.
Furthermore,
\begin{equation}\label{eq2Star}
    \max_{[\widehat b(N), \widehat L^+]} V
\leq
    V\big(\widehat b(N)\big)+(1-\varepsilon_1)\log N+a_0.
\end{equation}
Hence,
$$
    \max_{\widehat x_0\leq \ell\leq k\leq \widehat L^+-1,\ k\geq \widehat b(N)}[V(k)-V(\ell)]
\leq
    \max_{[\widehat b(N), \widehat L^+]} V
    -
    \min_{[\widehat x_0, \widehat L^+]} V
\leq
    (1-\varepsilon_1)\log N+2a_0.
$$
This, \eqref{InegEsperance1},  Markov's inequality and $\omega\in\Delta_N^{(3)}$ give
$$
    P_\omega^{\widehat b(N)}\big[\tau(\widehat x_0)\wedge \tau\big(\widehat L^+\big)> N^{1-\e_1/2}\big]
\leq
    N^{-(1-\e_1/2)}
    \varepsilon_0^{-1}4\e_3^{-2} (\log N)^4
    N^{1-\e_1}e^{2a_0}
\leq
    N^{-\e_1/4}
$$
uniformly for large $N$.
Moreover by \eqref{probaatteinte}, \eqref{InegV0}, \eqref{eq2Star} and since $\omega\in\Delta_N^{(3)}$,
$$
    P_\omega^{\widehat b(N)}\big[\tau(\widehat x_0)<\tau\big(\widehat L^+\big)\big]
\leq
    \big(\widehat L^+-\widehat b(N)\big)\exp\big[\max_{[\widehat b(N), \widehat L^+]} V-V\big(\widehat x_0\big)\big]
\leq
    \frac{(\log N)^2 e^{2a_0}}{\e_3N^{2\e_1}}
\leq
    \frac{1}{N^{\e_1/4}}
$$
uniformly for large $N$.
Consequently,
\begin{eqnarray*}
    Q_\omega\big[\tau\big(\widehat L^+\big)> N^{1-\e_1/2}\big]
& = &
    P_\omega^{\widehat b(N)}\big[\tau\big(\widehat L^+\big)> N^{1-\e_1/2}\big]
\\
& \leq &
    P_\omega^{\widehat b(N)}\big[\tau(\widehat x_0)\wedge \tau\big(\widehat L^+\big)> N^{1-\e_1/2}\big]
    +
    P_\omega^{\widehat b(N)}\big[\tau(\widehat x_0)<\tau\big(\widehat L^+\big)\big]
\\
& \leq &
    2 N^{-\e_1/4}.
\end{eqnarray*}
We prove similarly that
$
    Q_\omega\big[\tau\big(\widehat L^-\big)> N^{1-\e_1/2}\big]
\leq
    2 N^{-\e_1/4}
$ uniformly for large $N$,
using \eqref{InegEsperance2} and \eqref{InegVx2} instead of \eqref{InegEsperance1} and \eqref{InegV0}
respectively,
and  because $\min_{[\widehat x_0, \widehat x_2]} V\geq V\big(\widehat b(N)\big)-a_0$
which we proved after \eqref{InegMinVThetaX2}.
This proves Lemma \ref{LemInegaliteSupAtteinteLpm}.
\end{proof}

\begin{lem} \label{LemMajorationNuBordsVallee}
For large $N$,
$$
    \forall \omega \in \Delta_N \cap \Delta_N^{(R)},
\qquad
    \widehat \nu\big(\big[\widehat x_0, \widehat L^-\big]\big)
    +
    \widehat \nu\big(\big[\widehat L^+, \widehat x_2\big]\big)
\leq
    N^{-\e_1/4}.
$$
\end{lem}

\begin{proof}
Let $N\geq N_2$ and $\omega \in \Delta_N \cap \Delta_N^{(R)}$.
Recall that $\widehat x_0\leq \widehat L^-< \widehat b(N)< \widehat L^+\leq \widehat x_2$,
which is proved before \eqref{eqInfVentrex0et0}.
Notice that $\widehat L^-\leq x_1\leq \widehat L^+$, which is proved similarly as $x_1\leq \theta_N^{(R)}$ after  \eqref{InegPreuvex2bN}.
Using the same method as for \eqref{InegMinVThetaX2} with $\widehat L^+$ instead of $\theta_N^{(R)}$,
we get
$V\geq V\big(\widehat b(N)\big)+(\e_1/3)\log N$ on $\big[\widehat L^+, \widehat x_2\big]$.
Also, $V\big(\widehat L^+-1\big)\geq V\big(\widehat b(N)\big)+(\e_1/3)\log N$
Since $\widehat \mu(2\Z)\geq e^{-V(\widehat b(N))}$, this leads to
$$
    \widehat \nu\big(\big[\widehat L^+, \widehat x_2\big]\big)
\leq
    \big[\widehat x_2-\widehat L^++2\big]
    e^{-V(\widehat b(N))}N^{-\e_1/3}
    /{\widehat \mu(2\Z)}
\leq
    3\e_3^{-1}(\log N)^2
    N^{-\e_1/3}
\leq
    N^{-\e_1/4}/2
$$
uniformly for large $N$, where we used $\omega \in \Delta_N^{(3)}$. We prove similarly that
$
    \widehat \nu\big(\big[\widehat x_0, \widehat L^-\big]\big)
\leq
    N^{-\e_1/4}/2
$
uniformly for large $N$, which ends the proof of the lemma.
\end{proof}
\begin{lem}\label{Lemma_Time_Coupling}
There exists $N_4\in\N$ such that for $N\ge N_4$ for every $\omega \in \Delta_N \cap \Delta_N^{(R)}$,
\begin{equation}\label{InegProbaRencontre}
    Q_\omega\big[\tau_{\widehat Z=Z}>N^{1-\e_1/2}\big]
\leq
    5 N^{-\e_1/4}
\end{equation}
and
\begin{equation}\label{InegProbaExit}
    Q_\omega[\tau_{exit}\leq N]
\leq
    Q_\omega\big[\tau(\widehat x_0)\wedge \tau(\widehat x_2)< N\big]
=
    P_\omega^{\widehat b(N)}\big[\tau(\widehat x_0)\wedge \tau(\widehat x_2)< N\big]
\leq
    2e^{2a_0}N^{-\e_1}.
\end{equation}
\end{lem}
\begin{proof}
Due to Lemma \ref{LemInegaliteSupAtteinteLpm}, we have
for large $N$ for all $\omega \in \Delta_N \cap \Delta_N^{(R)}$,
\begin{eqnarray*}
&&
    Q_\omega\big[\tau_{\widehat Z=Z}>N^{1-\e_1/2}\big]
\\
& \leq &
    Q_\omega\big[\tau\big(\widehat L^-\big)\vee\tau\big(\widehat L^+\big)<\tau_{\widehat Z=Z}\big]
    +
    Q_\omega\big[\tau\big(\widehat L^-\big)\vee\tau\big(\widehat L^+\big)>N^{1-\e_1/2}\big]
\\
& \leq &
    Q_\omega\big[\tau\big(\widehat L^-\big)<\tau_{\widehat Z=Z},\, \widehat Z_0 <\widehat b(N)\big]
    +
    Q_\omega\big[\tau\big(\widehat L^+\big)<\tau_{\widehat Z=Z},\, \widehat Z_0 \geq \widehat b(N)\big]
    +4 N^{-\e_1/4}.
\end{eqnarray*}
Notice that under $Q_\omega$, $Z_0=\widehat b(N)\in (2\Z)$ by \eqref{eqLawUnderQ} and \eqref{eqDefWidehatbN},
%since $\omega \in \Delta_N^{(R)}$,
and $\widehat Z_0\in(2\Z)$ by \eqref{eqLawUnderQ} and \eqref{eqDefWidehatNu}.
So the process $\big(\widehat Z_k-Z_k\big)_{k\in\N}$ starts at $\big(\widehat Z_0-\widehat b(N)\big)\in(2\Z)$
and only makes jumps belonging to $\{-2,0,2\}$,
and thus up to time $\tau_{\widehat Z=Z}-1$ it is $<0$ (resp. $>0$) on
$\big\{\widehat Z_0 <\widehat b(N)\big\}$ \big(resp. $\big\{\widehat Z_0 \geq \widehat b(N)\big\}$\big),
and in particular at time $\tau\big( \widehat L^-\big)$ on
$\big\{\tau\big(\widehat L^-\big)<\tau_{\widehat Z=Z},\, \widehat Z_0 <\widehat b(N)\big\}$
\big(resp. at time $\tau\big( \widehat L^+\big)$ on
$\big\{\tau\big(\widehat L^+\big)<\tau_{\widehat Z=Z},\, \widehat Z_0 \geq \widehat b(N)\big\}$\big).
This gives for large $N$ for all $\omega \in \Delta_N \cap \Delta_N^{(R)}$,
\begin{eqnarray*}
&&
    Q_\omega\big[\tau_{\widehat Z=Z}>N^{1-\e_1/2}\big]
\\
& \leq &
    Q_\omega\big[\tau\big(\widehat L^-\big)<\tau_{\widehat Z=Z},\, \widehat Z_{\tau(\widehat L^-)}< \widehat L^-\big]
    +
    Q_\omega\big[\tau\big(\widehat L^+\big)<\tau_{\widehat Z=Z},\, \widehat Z_{\tau(\widehat L^+)} > \widehat L^+\big]
    +4 N^{-\e_1/4}
\\
& \leq &
    Q_\omega\big[\tau\big(\widehat L^-\big)<\tau_{\widehat Z=Z},\, \widehat Z_{2\lfloor\tau(\widehat L^-)/2\rfloor} \leq \widehat L^-\big]
    +
    Q_\omega\big[\tau\big(\widehat L^+\big)<\tau_{\widehat Z=Z},\, \widehat Z_{2\lfloor \tau(\widehat L^+)/2\rfloor} \geq \widehat L^+\big]
\\
&&
    +4 N^{-\e_1/4}
\\
& \leq &
    \widehat \nu\big(\big[\widehat x_0, \widehat L^-\big]\big)
    +
    \widehat \nu\big(\big[\widehat L^+, \widehat x_2\big]\big)
    +4 N^{-\e_1/4},
\end{eqnarray*}
where the last inequality comes from the fact that
%$\widehat Z$ is independent of $Z$ up to time $\tau_{\widehat Z=Z}$ under $Q_\omega$, and
$
    Q_{\omega}\big(\widehat Z_{2k}=x\big)
=
    P_{\widehat \omega}^{\widehat \nu}\big(\widehat Z_{2k}=x\big)
=
    \widehat \nu(x)
$
for every $x\in\Z$ and every (deterministic) $k\in\N$ as explained after \eqref{eqDefWidehatNu},
and from the independence of $\widehat Z$ with $Z$ and then $\tau(.)$ up to $\tau_{\widehat Z=Z}$.
Now, applying Lemma \ref{LemMajorationNuBordsVallee},
this gives \eqref{InegProbaRencontre} for large $N$ for every $\omega \in \Delta_N \cap \Delta_N^{(R)}$.

Due to \eqref{eqLawUnderQ} and Lemma \ref{LemmaProaResterDansVallee},
for large $N$ for every $\omega \in \Delta_N \cap \Delta_N^{(R)}$, \eqref{InegProbaExit} holds.
\end{proof}

\begin{proof}[Proof of Proposition \ref{Lemmeannexe1Bis}]
Recall that we have fixed $\delta\in(0,1)$ and that
\eqref{eqProbaLiminfBN} comes from Lemma \ref{LemmaProbaDeltaN}.
Let us prove \eqref{PZ=b}.
To this aim, we fix $(y_1, \dots, y_r)\in (2\Z)^r$.
Let $N_1\in\N$ be such that $N_1\geq\max(N_2,N_3,N_4)$
and such that for every $N\geq N_1$,
$
    \e_3^{-1}(\log N)^2 [N^{-\e_5/2}+2N^{-2\e_1}]
\leq
    \e_6^{-1}
$,
$
    N^{-(\e_1\wedge \e_2\wedge\e_4)/4}
\leq
    1/8
$,
$
    N^{1-\e_1/3}
\geq
    N^{1-\e_1}+N^{1-\e_1/2}
$
and
$
    5 N^{-\e_1/4}+2e^{2a_0}N^{-\e_1}
\leq
    \e_6e^{-a_0}/6
$,
recalling $a_0=\log((1-\varepsilon_0)/\varepsilon_0)$.
Now, we would like to give a lower bound for $P_\omega^{y_j}\big[Z_n=\widehat b(N)\big]$ for $n$ even. Recall \eqref{Deltandef} and
\eqref{DeltaNRdef}.
Let $N\geq N_1$,  $\omega \in \Delta_N \cap \Delta_N^{(R)}$,
$j\in\{1,\dots,r\}$,
and $n\in(2\N)$, with $n\in[N^{1-\e_1}+N^{1-\e_1/2},N]$.
We have by the strong Markov property,
\begin{eqnarray}
    P_\omega^{y_j}\big[Z_n=\widehat b(N)\big]
& \geq &
    P_\omega^{y_j}\big[Z_n=\widehat b(N), \tau\big(\widehat b(N)\big)\leq N^{1-\e_1}\big]
\nonumber\\
& = &
    E_\omega^{y_j}\big[{\bf 1}_{\{\tau(\widehat b(N))\leq N^{1-\e_1}\}}
        P_\omega^{\widehat b(N)}\big(Z_k=\widehat b(N)\big)_{|k=n-\tau(\widehat b(N))} \big]
\nonumber\\
& \geq &
    P_\omega^{y_j}\big[\tau\big(\widehat b(N)\big)\leq N^{1-\e_1}\big]
    \inf_{k\in[N^{1-\e_1/2}, N]\cap(2\N)}P_\omega^{\widehat b(N)}\big(Z_k=\widehat b(N)\big)\nonumber\\
&\geq &
    \left(1- N^{-(\e_1\wedge \e_2)/4}-N^{-\varepsilon_1/4}\right)
    \inf_{k\in[N^{1-\e_1/2}, N]\cap(2\N)}P_\omega^{\widehat b(N)}\big(Z_k=\widehat b(N)\big)
    \hphantom{aaa}
\label{InegProbaPyjZn}
\end{eqnarray}
because $\widehat b(N)$ and $y_j$ are even (see \eqref{eqDefWidehatbN})
and then $\tau\big(\widehat b(N)\big)$ is also even under $P_\omega^{y_j}$, and where
we used Lemma \ref{LemmaTempsAtteintebN} in the last line.
Moreover, for $k\in[N^{1-\e_1/2}, N]\cap(2\N)$,
\begin{eqnarray}
    P_\omega^{\widehat b(N)}\big(Z_k=\widehat b(N)\big)
& = &
    Q_\omega\big(Z_k=\widehat b(N)\big)
\nonumber\\
& \geq &
    Q_\omega\big(Z_k=\widehat b(N), \tau_{\widehat Z=Z}\leq N^{1-\e_1/2}, \tau_{exit}> N\big)
\nonumber\\
& = &
    Q_\omega\big(\widehat Z_k=\widehat b(N), \tau_{\widehat Z=Z}\leq N^{1-\e_1/2}, \tau_{exit}> N\big)
\nonumber\\
& \geq &
    Q_\omega\big(\widehat Z_k=\widehat b(N)\big) -Q_\omega\big(\tau_{\widehat Z=Z}> N^{1-\e_1/2}\big)-Q_\omega\big( \tau_{exit}\leq N\big)
\nonumber\\
& \geq &
    \widehat\nu \big(\widehat b(N)\big) -5 N^{-\e_1/4}-2e^{2a_0}N^{-\e_1},
\label{InegProbaZhatBhatk}
\end{eqnarray}
where we used \eqref{eqLawUnderQ} in the first and last line,
$Z_k=\widehat Z_k$ for $k\in\big[\tau_{\widehat Z=Z}, \tau_{exit}\big)$  under $Q_\omega$ in the third line,
and $Q_\omega\big(\widehat Z_k=x\big)=P_{\widehat \omega}^{\widehat \nu}\big(\widehat Z_k=x\big)=\widehat \nu(x)$
since $k$ is even,
\eqref{InegProbaRencontre} and \eqref{InegProbaExit} in the last line since $N\geq N_4$.

Notice that
$
    \widehat \mu(2\Z)
=
    e^{-V(\widehat b(N))}
    \sum_{i=\widehat x_0}^{\widehat x_2-1}e^{-[V(i)-V(\widehat b(N))]}
$,
with
$$
    \sum_{i=\widehat x_0}^{-1}e^{-[V(i)-V(\widehat b(N))]}
\leq
    |\widehat x_0| N^{-\e_5/2}
\leq
    \e_3^{-1}(\log N)^2 N^{-\e_5/2}
\leq
    \e_6^{-1}
$$
since $N\geq N_1$, $\omega\in\Delta_N^{(3)}$
and thanks to \eqref{eqInfVentrex0et0}.

Moreover, by \eqref{InegMinVThetaX2},
$
    \sum_{i=\theta_N^{(R)}}^{\widehat x_2-1}e^{-[V(i)-V(\widehat b(N))]}
\leq
    2\e_3^{-1}(\log N)^2 N^{-2\e_1}
\leq \e_6^{-1}
$ because $N\geq N_1$. Finally,
$
        \sum_{i=0}^{\theta_N^{(R)}-1} e^{-\big[V(i)-V\big(\beta_N^{(R)}\big)\big]}
        \leq
        \e_6^{-1}
$
since $\omega \in \Delta_N^{(6,R)}$ (see \eqref{Delta6}).
Moreover, $\big|V\big(\widehat b(N)\big)-V\big(\beta_N^{(R)}\big)\big|\leq a_0$.
Hence,
$
    \widehat \mu(2\Z)
\leq
    3\e_6^{-1}e^{a_0}
    e^{-V(\widehat b(N))}
$.
Moreover,
$\widehat \mu\big(\widehat b(N)\big)\geq e^{-V(\widehat b(N))}$
since $\widehat x_0< \widehat b(N)<\widehat x_2$,
and $\widehat b(N)$ is even by \eqref{eqDefWidehatbN}, so by \eqref{eqDefWidehatNu},
$
    \widehat \nu\big(\widehat b(N)\big)
=
    \widehat \mu\big(\widehat b(N)\big)/\widehat \mu(2\Z)
\geq
    \e_6e^{-a_0}/3
$.
This, \eqref{InegProbaPyjZn} and \eqref{InegProbaZhatBhatk} give
for $N\geq N_1$,
$$
    \forall \omega\in \Delta_N \cap \Delta_N^{(R)},
    \forall n\in\big[N^{1-\e_1/3},N\big]\cap(2\N),
    \forall j\in\{1,\dots ,r\},
\quad
    P_\omega^{y_j}\big[Z_n=\widehat b(N)\big]
\geq
    \e_6e^{-a_0}/8.
$$
The proof is similar for $\omega\in \Delta_N \cap \Delta_N^{(L)}$ by symmetry.
This, combined with Lemma \ref{LemmaProbaDeltaN},
ends \eqref{PZ=b} with $c(\delta)=\e_6e^{-a_0}/8>0$ and $\e(\delta)=\e_1/3$.
To prove that
this remains true if $(2\Z)^r$ and $2\N$ are replaced respectively by
$(2\Z+1)^r$ and $2\N+1$,
we just condition
$P_\omega^{y_j}\big[Z_n=\widehat b(N)\big]$
by $Z_1$, and apply the Markov property and \eqref{PZ=b}
to $(y_1\pm 1,\dots,y_r\pm 1)$.
\end{proof}
{\bf Acknowledgement}
A part of this work was done while AD and NG were visiting Brest. We thank ANR MEMEMO~2 (ANR-10-BLAN-0125) and the
LMBA, University of Brest for its hospitality.
We are grateful to an anonymous referee for comments which helped improve the presentation of the paper.

%%%%%%%%%%%%%%%%%%%%%%%%%%%%%%%%%%%%%%%%%%%%%%%%%%%%%%%%%%%%%%%%%%%%%%%%%%%%%%%%%%%%%%%%
%                                                                                      %
%                                   BIBLIOGRAPHY                                       %
%                                                                                      %
%%%%%%%%%%%%%%%%%%%%%%%%%%%%%%%%%%%%%%%%%%%%%%%%%%%%%%%%%%%%%%%%%%%%%%%%%%%%%%%%%%%%%%%%

\bibliographystyle{alpha}

\end{document}